\documentclass[11pt,reqno]{amsart}
\usepackage[utf8]{inputenc}
\usepackage{lmodern} % 字体优化
\usepackage{amsmath,amssymb,amsfonts,amsthm,latexsym}
\usepackage{graphicx,multirow,enumerate}
\usepackage{tikz}
\usepackage{array}
\usepackage{booktabs}
\usepackage{longtable}
\usetikzlibrary{calc}
\usepackage{booktabs} 
\usepackage{float}
\usepackage{caption}
\usepackage{array}      % 表格控制
\usepackage{cellspace}  % 表格间距控制
\usepackage{mathrsfs}   % 花体字母
\usepackage[all]{xy}    % 交换图
\usepackage{color,xcolor} % 颜色
\usepackage{cancel}     % 删除线
\usepackage{url}        % 网址支持
\usepackage{longtable}
\usepackage[hidelinks]{hyperref}
\theoremstyle{plain}
\newtheorem{thm}{Theorem}[section]

\newtheorem{prop}[thm]{Proposition}

\theoremstyle{definition}

\newtheorem{defn}[thm]{Definition}
\newtheorem{rem}[thm]{Remark}

\newtheorem{conjecture}{Conjecture}[section]

\begin{document}

\title{Classification of Novikov-Poisson Algebras and Their Applications}

    \author{Ziyi Zhang}
	\address{School of Mathematics and Statistics, Northeast Normal University, Changchun 130024, Jilin, China}
	\email{zhangziyi@nenu.edu.cn}
	
	\author{Zeyu Hao}
	\address{School of Mathematics and Statistics, Northeast Normal University, Changchun 130024, Jilin, China}
	\email{haozeyu@nenu.edu.cn}
	
	\author{Yining Sun}
	\address{School of Mathematics and Statistics, Northeast Normal University, Changchun 130024, Jilin, China}
	\email{yiningsun@nenu.edu.cn}
	
	\author{Liangyun Chen*}
	\address{School of Mathematics and Statistics, Northeast Normal University, Changchun 130024, Jilin, China}
	\email{chenly640@nenu.edu.cn}
	\thanks{Corresponding author: Liangyun Chen.}
	\thanks{This work is supported by the NNSF of China (No.~12271085).}	
	
\begin{abstract}
In this paper, we give a complete classification, up to isomorphism, of 3-dimensional complex Novikov--Poisson algebras. As an application of the classification, we further prove that every 3-dimensional complex transposed Poisson algebra can be obtained from a Novikov--Poisson algebra except the Lie algebra $\mathfrak{sl}_2(\mathbb C)$. Consequently, Sartayev's conjecture holds in dimension 3, that is, every 3-dimensional complex transposed Poisson algebra is special when regarded as a GD algebra.
\end{abstract}


\maketitle

\section{Introduction}\label{introduction}

The classification, up to isomorphism, of $n$-dimensional algebras satisfying a prescribed family of polynomial identities is a classical problem in the theory of associative and nonassociative algebras.  Complete low-dimensional classifications have been obtained for many important varieties.  These include complex 5-dimensional nilpotent symmetric Leibniz algebras \cite{AlvarezKaygorodov2021}, 6-dimensional nilpotent anticommutative algebras \cite{KaygorodovKhrypchenkoLopes2020}, complex 4-dimensional symmetric Leibniz algebras, complex 4-dimensional nilpotent terminal algebras \cite{KaygorodovKhrypchenkoPopov2021}, complex 3-dimensional right alternative and semi-alternative algebras \cite{AbdelwahabKaygorodovLubkovRight2026}, and complex 3-dimensional metabelian commutative, derived commutative associative, derived Jordan,  bicommutative algebras \cite{AbdelwahabKaygorodovLubkovDerived2026}.

Novikov algebras arose from the study of Poisson brackets of hydrodynamic type and Hamiltonian operators \cite{Balinskii-Novikov,Gelfand-Dorfman}.  Their two defining identities combine left symmetry with right commutativity.  Xu introduced Novikov--Poisson algebras and their tensor theory in \cite{Xu1996} and subsequently developed their structure in \cite{Xu1997}.  The 2-dimensional Novikov--Poisson algebras were first classified in \cite{ZhaoBaiMeng}.  More recently, their geometric classification were studied in \cite{AbdelwahabFernandezKaygorodov2025}. Finite-dimensional simple Novikov--Poisson algebras over fields of characteristic different from \(2\) were determined in \cite{ZakharovSimple2023}.  Further work developed unified products and extending structures for Novikov--Poisson algebras \cite{BaoHong}, and established a bialgebra theory for Novikov--Poisson algebras \cite{SunZhuLi}.  In characteristic zero, every Novikov--Poisson algebra embeds into a commutative conformal algebra with a derivation \cite{Kolesnikov}.

Transposed Poisson algebras were first introduced in \cite{Bai}, where they were defined as a dual notion to Poisson algebras by exchanging the roles of the associative product and the Lie bracket in the Leibniz rule. Research on transposed Poisson algebras has gained significant attention in recent years. Transposed Poisson
algebra structures on simple Lie algebras were studied in \cite{FernandezOuaridi2024,FernandezOuaridi2026}, and the classification of complex \(3,4\)-dimensional transposed Poisson algebras was obtained in
\cite{BeitesFernandezKaygorodov,Yuan,Zhang-Z,Zhang4D}. Moreover, further progress has been made in the study of transposed Poisson structures on specific algebraic frameworks. By building the connection between
\(\frac12\)-derivations of Lie algebras and transposed Poisson algebras \cite{Ferreira}, descriptions of all transposed Poisson structures on certain Lie algebras have been obtained. These include the Witt algebra, the Virasoro algebra and the \(W(a,b)\) algebra \cite{Ferreira}, Block Lie algebras and superalgebras \cite{Kaygorodov-Block}, solvable Lie algebras with filiform nilradical \cite{Abdurasulov-Adashev-Eshmeteva}, Galilean and solvable Lie algebras \cite{Kaygorodov-Lopatkin-Zhang}, and other classes of Lie algebras.
Further developments include Hom and BiHom versions \cite{LaraiedhSilvestrov2021,MaLi2023}, and the recent study of nilpotency and Frattini theory for transposed Poisson algebras \cite{JinHong2026}.

Gelfand--Dorfman (GD) algebras originated in the study of Hamiltonian operators and Hamiltonian pairs in integrable systems \cite{Gelfand-Dorfman,Dorfman1993}.  A GD algebra carries a Novikov product and a Lie bracket satisfying a compatibility identity.  Xu related GD superalgebras to quadratic Lie conformal superalgebras \cite{Xu2000} and constructed GD algebras from classical $R$-matrices \cite{Xu2002}.  A standard
construction starts from a differential Poisson algebra.  A GD algebra is called special if it embeds into an algebra obtained by this construction.  Kolesnikov, Sartayev and Orazgaliev exhibited two independent special identities of degree four \cite{Kolesnikov-Sartayev-Orazgaliev}.  Kolesnikov and Sartayev proved that all
special identities of degree at most five follow from them, while it remains open whether they generate all special identities \cite{Kolesnikov-Sartayev}.  The relation with Novikov--Poisson algebras was further clarified in \cite{Kolesnikov}, where it was shown that every commutator GD algebra (known as a transposed Poisson algebra) obtained from a Novikov--Poisson algebra is special. Free special GD algebras were studied in
\cite{Gubarev-Sartayev}.  Recently, nilpotency, solvability, constructions of special and non-special GD algebras, GD structures on simple Lie algebras, and the complete classification of complex GD algebras in dimensions 2 and 3 were studied in \cite{ZhangGD2026}, where the speciality of the
classified algebras was also determined. The cohomology, deformations and inducibility problems for GD algebras were also studied in \cite{GDCohomology}.

The three kinds of algebras above are linked in two direct ways.  First, the commutator of the Novikov product in a Novikov--Poisson algebra, together with its commutative product, defines a transposed Poisson algebra \cite{Bai}. Second, every transposed Poisson algebra becomes a GD algebra when its commutative associative product is regarded as the Novikov product \cite{Sartayev}.  The first connection raises the question of which transposed Poisson algebras can be obtained in this way.  A transposed Poisson algebra is called special, if it can be obtained from a Novikov--Poisson algebra \cite{BeitesFernandezKaygorodov}. The special transposed Poisson algebras in dimension 2 were studied in \cite{BeitesFernandezKaygorodov}. These results are part of the broader study of low-dimensional transposed Poisson algebras. Complete classifications of complex transposed Poisson algebras in dimensions 3 and 4 are now available \cite{BeitesFernandezKaygorodov,Zhang-Z,Zhang4D}. These low-dimensional classifications motivate us to study the speciality of 3-dimensional transposed Poisson algebras.  For this purpose, we classify the 3-dimensional complex Novikov--Poisson algebras. Using our classification of 3-dimensional Novikov--Poisson algebras, we construct the associated transposed Poisson algebras and establish their isomorphism correspondence with the existing classification.  This proves that every 3-dimensional transposed Poisson algebra other than $\mathfrak{sl}_2(\mathbb C)$ is special.

On the GD side, the speciality problem is a central topic.  A GD algebra is called special if it can be embedded into a GD algebra induced by a differential Poisson algebra.  Kolesnikov and Sartayev proved that every 2-dimensional GD algebra is special \cite{Kolesnikov-Sartayev}.  Our recent classification of complex 3-dimensional GD algebras contains both special and non-special classes \cite{ZhangGD2026}, so speciality is not automatic for arbitrary 3-dimensional GD algebras.  Sartayev nevertheless conjectured that every transposed Poisson algebra is special when regarded as a GD algebra \cite{Sartayev}. This conjecture was also included as Question~30 in the survey on transposed Poisson algebras \cite{BeitesStructures}.  Since constructing such embeddings is generally difficult, we use the fact that every transposed Poisson algebra obtained from a Novikov--Poisson algebra is special when regarded as a GD algebra \cite{Kolesnikov}.  Based on the above results, we prove that every 3-dimensional transposed Poisson algebra, regarded as a GD algebra, is special.  Consequently, Sartayev's conjecture that every transposed Poisson algebra is special as a GD algebra holds in dimension 3.

The paper is organized as follows. Section~\ref{sec:method} gives a complete classification of 3-dimensional Novikov--Poisson algebras over \(\mathbb C\).  Section~\ref{sec:construction} applies this classification to determine the speciality of 3-dimensional transposed Poisson algebras.  It proves that every such algebra except \(\mathfrak{sl}_2(\mathbb C)\) is special.  Based on these results, Section~\ref{sec:gd_conjecture} proves that every 3-dimensional transposed Poisson algebra is special as a GD algebra and hence establishes
Sartayev's conjecture in dimension 3.

Throughout this paper, unless otherwise stated, all vector spaces and algebras are assumed to be finite-dimensional over $\mathbb C$.

	\section{The classification of low-dimensional Novikov-Poisson algebras}\label{sec:method}
	 In this section, we first recall the definitions and results used below. Then we present the methodology for classifying low-dimensional Novikov-Poisson algebras and classify the 3-dimensional Novikov-Poisson algebras over $\mathbb{C}$.

A \emph{Novikov algebra} is a vector space with a bilinear product $\circ$ satisfying, for all $x,y,z$,
\begin{align}
(x\circ y)\circ z-x\circ(y\circ z)
  &=(y\circ x)\circ z-y\circ(x\circ z),\label{eq:novikov-left-symmetric}\\
(x\circ y)\circ z&=(x\circ z)\circ y.\label{eq:novikov-right-commutative}
\end{align}

\begin{defn}\cite{Xu1997}
A \emph{Novikov--Poisson algebra} is a triple $(A,\cdot,\circ)$ in which $\cdot$ and $\circ$ are bilinear operations on $A$ such that
\begin{enumerate}
    \item $(A,\cdot)$ is a commutative associative algebra
    \item $(A,\circ)$ is a Novikov algebra
    \item the compatibility identities
    \begin{align}
        (x\cdot y)\circ z&=x\cdot(y\circ z),\label{eq:NP3}\\
        (x\circ y)\cdot z-(y\circ x)\cdot z
        &=x\circ(y\cdot z)-y\circ(x\cdot z)\label{eq:NP4}
    \end{align}
    hold for all $x,y,z\in A$.
\end{enumerate}
\end{defn}

	\begin{defn}\cite{BaoHong}
Let \((A_1,\cdot_1,\circ_1)\) and \((A_2,\cdot_2,\circ_2)\) be two Novikov--Poisson algebras. They are said to be isomorphic if there exists a linear isomorphism \(\varphi:A_1\to A_2\) such that, for all \(x,y\in A_1\),
\begin{equation*}
        \varphi(x\cdot_1 y)=\varphi(x)\cdot_2\varphi(y),
        \qquad
        \varphi(x\circ_1 y)=\varphi(x)\circ_2\varphi(y).
\end{equation*}
\end{defn}

\begin{defn}
Let \((A,\cdot)\) be a commutative associative algebra. Denote by \(\mathcal N_{\mathrm{NP}}(A,\cdot)\) the set of all bilinear maps \(\theta:A\times A\to A\) satisfying, for all \(x,y,z\in A\),
\begin{align*}
\theta(x,\theta(y,z))-\theta(\theta(x,y),z)
&=
\theta(y,\theta(x,z))-\theta(\theta(y,x),z), \\
\theta(\theta(x,y),z)
&=
\theta(\theta(x,z),y), \\
\theta(x\cdot y,z)
&=
x\cdot\theta(y,z), \\
\theta(x,y)\cdot z-\theta(y,x)\cdot z
&=
\theta(x,y\cdot z)-\theta(y,x\cdot z).
\end{align*}
\end{defn}

\begin{prop}
If $\theta \in \mathcal N_{\mathrm{NP}}(A,\cdot)$ and we define a bilinear product $\circ_\theta : A \times A \to A$ by
\[ 
x \circ_\theta y := \theta(x, y) ,\quad \text{for all } x, y \in A. 
\]

Then the triple $(A, \cdot, \circ_\theta)$ carries the structure of a Novikov-Poisson algebra.

Conversely, if $(A, \cdot, \circ)$ is a Novikov-Poisson algebra, the bilinear map $\theta : A \times A \to A$ defined by $\theta(x, y) := x \circ y$, then $\theta \in \mathcal N_{\mathrm{NP}}(A,\cdot)$, and the Novikov-Poisson algebra $(A, \cdot, \circ)$ is isomorphic to $(A, \cdot, \circ_\theta)$.

\end{prop}
\begin{proof}
	The proof is straightforward and therefore omitted.
\end{proof}

\begin{rem}
\label{rem:first}
Let \((A,\cdot)\) be an \(n\)-dimensional commutative associative algebra with basis \(\{e_1,e_2,\dots,e_n\}\). For \(\theta\in \mathcal N_{\mathrm{NP}}(A,\cdot)\), define linear maps \(L_{e_i}:A\to A\) by \(L_{e_i}(x)=\theta(e_i,x)\), for \(i=1,2,\dots,n\). More generally, for \(x=\sum_{i=1}^n k_i e_i\), set
\(L_x=\sum_{i=1}^n k_iL_{e_i}\). It is straightforward to verify that \((L_{e_1},\dots,L_{e_n})\in\mathcal N'_{\mathrm{NP}}(A,\cdot)\), where \(\mathcal N'_{\mathrm{NP}}(A,\cdot)\) is defined as follows:
\[
\mathcal N'_{\mathrm{NP}}(A,\cdot)
:=
\left\{
(L_{e_1},\dots,L_{e_n})
\in \operatorname{End}(A)^n
\ \middle|\ 
\begin{aligned}
& L_{e_i}(L_{e_j}(e_k))-L_{L_{e_i}(e_j)}(e_k) \\
&\qquad =
  L_{e_j}(L_{e_i}(e_k))-L_{L_{e_j}(e_i)}(e_k), \\
& L_{L_{e_i}(e_j)}(e_k)
  =
  L_{L_{e_i}(e_k)}(e_j), \\
& L_{e_i\cdot e_j}(e_k)
  =
  e_i\cdot L_{e_j}(e_k), \\
& L_{e_i}(e_j)\cdot e_k
  -
  L_{e_j}(e_i)\cdot e_k \\
&\qquad =
  L_{e_i}(e_j\cdot e_k)
  -
  L_{e_j}(e_i\cdot e_k), \\
& \forall i,j,k\in\{1,\dots,n\}
\end{aligned}
\right\}.
\]
Conversely, let \((L_{e_1},\dots,L_{e_n})\in \mathcal N'_{\mathrm{NP}}(A,\cdot)\), and define the bilinear map
\(\theta:A\times A\to A\) by
\[
        \theta(x,y)=\sum_{i=1}^n k_iL_{e_i}(y),
        \qquad
        x=\sum_{i=1}^n k_i e_i.
\]
Then we can easily check that \(\theta\in\mathcal N_{\mathrm{NP}}(A,\cdot)\). Thus, after fixing the basis \(\{e_1,\dots,e_n\}\), the set \(\mathcal N_{\mathrm{NP}}(A,\cdot)\) is equivalent to \(\mathcal N'_{\mathrm{NP}}(A,\cdot)\).

Therefore, determining \((L_{e_1},L_{e_2},\dots,L_{e_n})\) amounts to solving the following system of polynomial equations:
\begin{equation}
\label{eq:NP-polynomial-system}
\begin{aligned}
& L_{e_i}(L_{e_j}(e_k))-L_{L_{e_i}(e_j)}(e_k)
 =
  L_{e_j}(L_{e_i}(e_k))-L_{L_{e_j}(e_i)}(e_k), \\
& L_{L_{e_i}(e_j)}(e_k)=L_{L_{e_i}(e_k)}(e_j), \\
& L_{e_i\cdot e_j}(e_k)=e_i\cdot L_{e_j}(e_k), \\
& L_{e_i}(e_j)\cdot e_k-L_{e_j}(e_i)\cdot e_k
 =
  L_{e_i}(e_j\cdot e_k)-L_{e_j}(e_i\cdot e_k),
\end{aligned}
\qquad \forall i,j,k\in\{1,\dots,n\}.
\end{equation}

Solving this system yields the collection of linear maps \((L_{e_1},L_{e_2},\dots,L_{e_n})\), and hence determines all Novikov--Poisson algebra structures associated with the associative algebra \((A,\cdot)\).
\end{rem}

\begin{prop}\label{prop:orbit-criterion}
Let \((A,\cdot)\) be a commutative associative algebra, and let \(\theta_1,\theta_2\in \mathcal N_{\mathrm{NP}}(A,\cdot)\). Then the Novikov--Poisson algebras \((A,\cdot,\circ_{\theta_1})\) and \((A,\cdot,\circ_{\theta_2})\) are isomorphic if and only if there exists \(\varphi\in\operatorname{Aut}(A,\cdot)\) such that
\[
        \varphi(\theta_1(x,y))
        =
        \theta_2(\varphi(x),\varphi(y)),
        \qquad \forall x,y\in A.
\]
\end{prop}
\begin{proof}
	The proof is straightforward and therefore omitted.
\end{proof}

\begin{rem}
\label{rem:second}
Let \((A,\cdot)\) be an \(n\)-dimensional commutative associative algebra with a fixed basis \(\{e_1,e_2,\dots,e_n\}\), and let \(\theta_1,\theta_2\in\mathcal N_{\mathrm{NP}}(A,\cdot)\). For \(i=1,2,\dots,n\), define the linear operators \(L^1_{e_i},L^2_{e_i}:A\to A\) by
\[
        L^1_{e_i}(x)=\theta_1(e_i,x),
        \qquad
        L^2_{e_i}(x)=\theta_2(e_i,x).
\]
We denote by the same symbols their respective matrix representations with respect to the fixed basis. Then the Novikov--Poisson algebras \((A,\cdot,\circ_{\theta_1})\) and \((A,\cdot,\circ_{\theta_2})\) are isomorphic if and only if there exists an associative algebra automorphism \(\varphi\in\operatorname{Aut}(A,\cdot)\), represented by the invertible matrix
\[
\varphi=
\begin{pmatrix}
x_{11} & x_{12} & \cdots & x_{1n} \\
x_{21} & x_{22} & \cdots & x_{2n} \\
\vdots & \vdots & \ddots & \vdots \\
x_{n1} & x_{n2} & \cdots & x_{nn}
\end{pmatrix},
\]
such that, for each \(i=1,2,\dots,n\),
\begin{equation}
\label{eq:NP-isomorphism-matrix}
        \bigl(
        x_{1i}L^2_{e_1}
        +x_{2i}L^2_{e_2}
        +\cdots
        +x_{ni}L^2_{e_n}
        \bigr)\varphi
        =
        \varphi L^1_{e_i}.
\end{equation}
\end{rem}

\begin{thm}
Let \((A_1,\cdot_1)\) and \((A_2,\cdot_2)\) be isomorphic commutative associative algebras, and let \(\theta_1\in \mathcal N_{\mathrm{NP}}(A_1,\cdot_1)\). Then there exists \(\theta_2\in \mathcal N_{\mathrm{NP}}(A_2,\cdot_2)\) such that the Novikov--Poisson algebras \((A_1,\cdot_1,\circ_{\theta_1})\) and \((A_2,\cdot_2,\circ_{\theta_2})\) are isomorphic.
\end{thm}

\begin{proof}
Let \(f:A_1\to A_2\) be an isomorphism of commutative associative algebras.
Define
\[
        \theta_2(x,y)
        =
        f\bigl(\theta_1(f^{-1}(x),f^{-1}(y))\bigr),
        \qquad x,y\in A_2.
\]
Then
\[
        \theta_2(f(a),f(b))=f(\theta_1(a,b)),
        \qquad a,b\in A_1.
\]
Thus the product \(\circ_{\theta_2}\) is the transport of \(\circ_{\theta_1}\) by \(f\). Since \(f\) also preserves the commutative associative products, the defining identities of Novikov--Poisson algebras are
preserved under \(f\). Hence \(\theta_2\in \mathcal N_{\mathrm{NP}}(A_2,\cdot_2)\), and \(f\) is an
isomorphism from \((A_1,\cdot_1,\circ_{\theta_1})\) to \((A_2,\cdot_2,\circ_{\theta_2})\).
\end{proof}

Solving the polynomial system~\eqref{eq:NP-polynomial-system} to determine all Novikov--Poisson algebra structures on a given commutative associative algebra yields structures depending on multiple parameters. Using the techniques developed in \cite{Zhang-Z}, we simplify the classification by minimizing the parameter set up to isomorphism through \textbf{isomorphism reduction}. Subsequently, we determine whether the structures are pairwise isomorphic, which we refer to as \textbf{isomorphism testing}. Details of isomorphism reduction and isomorphism testing are given in \cite{Zhang-Z}.

Below we present the classification of 3-dimensional commutative associative algebras.

\begin{thm}\cite{Rakhimov}\label{thm:class}
Let $(A, \cdot)$ be a $3$-dimensional commutative associative algebra with nonzero multiplication. 
Then $(A, \cdot)$ is isomorphic to one of the following algebras, where $\{e_1, e_2, e_3\}$ is a basis of $A$ and the omitted products are zero:
\[
\renewcommand{\arraystretch}{1.2}
\begin{array}{rl}
    \mathrm{A}_{1,1} : & e_1 \cdot e_3 = e_2. \\
    \mathrm{A}_{1,2} : & e_1 \cdot e_1 = e_2, \quad  e_1 \cdot e_2 = e_3. \\
    \mathrm{A}_{1,3} : & e_1 \cdot e_3 = e_1, \quad  e_2 \cdot e_3 = e_2, \quad  e_3 \cdot e_3 = e_3. \\
    \mathrm{A}_{1,4} : & e_1 \cdot e_3 = e_2, \quad  e_2 \cdot e_3 = e_2, \quad  e_3 \cdot e_3 = e_3. \\
    \mathrm{A}_{1,5} : & e_1 \cdot e_1 = e_2, \quad  e_1 \cdot e_3 = e_1, \quad  e_2 \cdot e_3 = e_2, \quad  e_3 \cdot e_3 = e_3. \\
    \mathrm{A}_{1,6} : & e_1 \cdot e_1 = e_1, \quad  e_2 \cdot e_2 = e_2, \quad  e_3 \cdot e_3 = e_3. \\
    \mathrm{A}_{1,7} : & e_1 \cdot e_1 = e_1, \quad  e_2 \cdot e_2 = e_2, \quad  e_1 \cdot e_3 = e_3. \\
    \mathrm{A}_{1,8} : & e_1 \cdot e_1 = e_1, \quad  e_2 \cdot e_2 = e_2. \\
    \mathrm{A}_{1,9} : & e_1 \cdot e_1 = e_1, \quad  e_2 \cdot e_2= e_3. \\
    \mathrm{A}_{1,10} : & e_1 \cdot e_1 = e_1. \\
    \mathrm{A}_{1,11} : & e_1 \cdot e_1 = e_2.
\end{array}
\]
\end{thm}

\begin{thm}\label{thm:classification-np3}
		
		Let \((A, \cdot, \circ)\) be a nonzero complex 3-dimensional Novikov-Poisson algebra, then \((A, \cdot, \circ)\) is isomorphic to one of the following algebras:

\footnotesize

\begin{enumerate}

	\item \textbf{NP01$_{\alpha,\beta}$:}
	$
	\begin{cases}
		e_1 \cdot e_3 = e_2 \\
		e_1 \circ e_1 = e_1, \quad e_1 \circ e_2 = -e_2, \quad e_1 \circ e_3 = e_1, \quad e_2 \circ e_1 = e_2, \quad e_2 \circ e_3 = e_2 \\
		e_3 \circ e_1 = \alpha e_2+e_3, \quad e_3 \circ e_2 = -e_2, \quad e_3 \circ e_3 = \beta e_2+e_3
	\end{cases}
	$

	\item \textbf{NP02$_{\alpha}$:}
	$
	\begin{cases}
		e_1 \cdot e_3 = e_2 \\
		e_1 \circ e_1 = e_1, \quad e_1 \circ e_2 = -e_2, \quad e_2 \circ e_1 = e_2, \quad e_3 \circ e_1 = \alpha e_2+e_3, \\
		e_3 \circ e_3 = e_2
	\end{cases}
	$

	\item \textbf{NP03$_{\alpha}$:}
	$
	\begin{cases}
		e_1 \cdot e_3 = e_2 \\
		e_1 \circ e_1 = e_1, \quad e_1 \circ e_2 = -e_2, \quad e_2 \circ e_1 = e_2, \quad e_3 \circ e_1 = \alpha e_2+e_3
	\end{cases}
	$

	\item \textbf{NP04$_{\alpha,\beta,\gamma}$:}
	$
	\begin{cases}
		e_1 \cdot e_3 = e_2 \\
		e_1 \circ e_1 = e_2, \quad e_1 \circ e_3 = \alpha e_2, \quad e_3 \circ e_1 = \beta e_2, \quad e_3 \circ e_3 = \gamma e_2
	\end{cases}
	$

	\item \textbf{NP05$_{\alpha,\beta}$:}
	$
	\begin{cases}
		e_1 \cdot e_3 = e_2 \\
		e_1 \circ e_3 = \alpha e_2, \quad e_3 \circ e_1 = \beta e_2
	\end{cases}
	$

	\item \textbf{NP06$_{\alpha,\beta}^{\alpha\neq 2}$:}
	$
	\begin{cases}
		e_1 \cdot e_1 = e_2, \quad e_1 \cdot e_2 = e_3 \\
		e_1 \circ e_1 = \alpha e_1, \quad e_1 \circ e_2 = e_2+\beta e_3, \quad e_1 \circ e_3 = (2-\alpha) e_3 \\
		e_2 \circ e_1 = \alpha e_2, \quad e_2 \circ e_2 = e_3, \quad e_3 \circ e_1 = \alpha e_3
	\end{cases}
	$

	\item \textbf{NP07$_{\alpha,\beta}$:}
	$
	\begin{cases}
		e_1 \cdot e_1 = e_2, \quad e_1 \cdot e_2 = e_3 \\
		e_1 \circ e_1 = 2e_1+\alpha e_3, \quad e_1 \circ e_2 = e_2+\beta e_3, \quad e_2 \circ e_1 = 2e_2 \\
		e_2 \circ e_2 = e_3, \quad e_3 \circ e_1 = 2e_3
	\end{cases}
	$

	\item \textbf{NP08$_{\alpha}$:}
	$
	\begin{cases}
		e_1 \cdot e_1 = e_2, \quad e_1 \cdot e_2 = e_3 \\
		e_1 \circ e_1 = e_1+\alpha e_2, \quad e_1 \circ e_3 = -e_3, \quad e_2 \circ e_1 = e_2+\alpha e_3, \quad e_3 \circ e_1 = e_3
	\end{cases}
	$

	\item \textbf{NP09$_{\alpha,\beta}$:}
	$
	\begin{cases}
		e_1 \cdot e_1 = e_2, \quad e_1 \cdot e_2 = e_3 \\
		e_1 \circ e_1 = \alpha e_2, \quad e_1 \circ e_2 = \beta e_3, \quad e_2 \circ e_1 = \alpha e_3
	\end{cases}
	$

	\item \textbf{NP10$_{\alpha}$:}
	$
	\begin{cases}
		e_1 \cdot e_1 = e_2, \quad e_1 \cdot e_2 = e_3 \\
		e_1 \circ e_1 = \alpha e_2+e_3, \quad e_1 \circ e_2 = \alpha e_3, \quad e_2 \circ e_1 = \alpha e_3
	\end{cases}
	$

	\item \textbf{NP11$_{\alpha,\beta,\gamma}$:}
	$
	\begin{cases}
		e_1 \cdot e_3 = e_1, \quad e_2 \cdot e_3 = e_2, \quad e_3 \cdot e_3 = e_3 \\
		e_1 \circ e_3 = \alpha e_1, \quad e_2 \circ e_3 = \alpha e_2, \quad e_3 \circ e_1 = e_2, \quad e_3 \circ e_2 = \beta e_1+\gamma e_2 \\
		e_3 \circ e_3 = e_1+\alpha e_3
	\end{cases}
	$

	\item \textbf{NP12$_{\alpha,\beta,\gamma}$:}
	$
	\begin{cases}
		e_1 \cdot e_3 = e_1, \quad e_2 \cdot e_3 = e_2, \quad e_3 \cdot e_3 = e_3 \\
		e_1 \circ e_3 = \alpha e_1, \quad e_2 \circ e_3 = \alpha e_2, \quad e_3 \circ e_1 = \beta e_1, \quad e_3 \circ e_2 = e_1+\gamma e_2 \\
		e_3 \circ e_3 = e_1+\alpha e_3
	\end{cases}
	$

	\item \textbf{NP13$_{\alpha,\beta}$:}
	$
	\begin{cases}
		e_1 \cdot e_3 = e_1, \quad e_2 \cdot e_3 = e_2, \quad e_3 \cdot e_3 = e_3 \\
		e_1 \circ e_3 = \alpha e_1, \quad e_2 \circ e_3 = \alpha e_2, \quad e_3 \circ e_1 = \beta e_1, \quad e_3 \circ e_2 = \beta e_2 \\
		e_3 \circ e_3 = e_1+\alpha e_3
	\end{cases}
	$

	\item \textbf{NP14$_{\alpha,\beta, \gamma}$:}
	$
	\begin{cases}
		e_1 \cdot e_3 = e_1, \quad e_2 \cdot e_3 = e_2, \quad e_3 \cdot e_3 = e_3 \\
		e_1 \circ e_3 = \alpha e_1, \quad e_2 \circ e_3 = \alpha e_2, \quad e_3 \circ e_1 = \beta e_2 \\
		e_3 \circ e_2 = e_1+\gamma e_2, \quad e_3 \circ e_3 = \alpha e_3
	\end{cases}
	$

	\item \textbf{NP15$_{\alpha,\beta}$:}
	$
	\begin{cases}
		e_1 \cdot e_3 = e_1, \quad e_2 \cdot e_3 = e_2, \quad e_3 \cdot e_3 = e_3 \\
		e_1 \circ e_3 = \alpha e_1, \quad e_2 \circ e_3 = \alpha e_2, \quad e_3 \circ e_1 = \beta e_1 \\
		e_3 \circ e_2 = \beta e_2, \quad e_3 \circ e_3 = \alpha e_3
	\end{cases}
	$

	\item \textbf{NP16$_{\alpha,\beta}$:}
	$
	\begin{cases}
		e_1 \cdot e_3 = e_2, \quad e_2 \cdot e_3 = e_2, \quad e_3 \cdot e_3 = e_3 \\
		e_1 \circ e_1 = -e_1+e_2, \quad e_1 \circ e_3 = \alpha e_2, \quad e_2 \circ e_3 = \alpha e_2, \quad e_3 \circ e_1 = \beta e_2 \\
		e_3 \circ e_2 = \beta e_2, \quad e_3 \circ e_3 = e_2+\alpha e_3
	\end{cases}
	$

	\item \textbf{NP17$_{\alpha,\beta}$:}
	$
	\begin{cases}
		e_1 \cdot e_3 = e_2, \quad e_2 \cdot e_3 = e_2, \quad e_3 \cdot e_3 = e_3 \\
		e_1 \circ e_3 = \alpha e_2, \quad e_2 \circ e_3 = \alpha e_2, \quad e_3 \circ e_1 = \beta e_2, \quad e_3 \circ e_2 = \beta e_2 \\
		e_3 \circ e_3 = e_2+\alpha e_3
	\end{cases}
	$

	\item \textbf{NP18$_{\alpha,\beta}$:}
	$
	\begin{cases}
		e_1 \cdot e_3 = e_2, \quad e_2 \cdot e_3 = e_2, \quad e_3 \cdot e_3 = e_3 \\
		e_1 \circ e_1 = -e_1+e_2, \quad e_1 \circ e_3 = \alpha e_2, \quad e_2 \circ e_3 = \alpha e_2 \\
		e_3 \circ e_1 = \beta e_2, \quad e_3 \circ e_2 = \beta e_2, \quad e_3 \circ e_3 = \alpha e_3
	\end{cases}
	$

	\item \textbf{NP19$_{\alpha,\beta}$:}
	$
	\begin{cases}
		e_1 \cdot e_3 = e_2, \quad e_2 \cdot e_3 = e_2, \quad e_3 \cdot e_3 = e_3 \\
		e_1 \circ e_3 = \alpha e_2, \quad e_2 \circ e_3 = \alpha e_2, \quad e_3 \circ e_1 = \beta e_2 \\
		e_3 \circ e_2 = \beta e_2, \quad e_3 \circ e_3 = \alpha e_3
	\end{cases}
	$

	\item \textbf{NP20$_{\alpha,\beta,\gamma}$:}
	$
	\begin{cases}
		e_1 \cdot e_1 = e_2, \quad e_1 \cdot e_3 = e_1, \quad e_2 \cdot e_3 = e_2, \quad e_3 \cdot e_3 = e_3 \\
		e_1 \circ e_1 = (\frac{\alpha+\beta}{2}) e_2, \quad e_1 \circ e_3 = \beta e_1+e_2, \quad e_2 \circ e_3 = \beta e_2 \\
		e_3 \circ e_1 = (\frac{\alpha+\beta}{2}) e_1+\gamma e_2, \quad e_3 \circ e_2 = \alpha e_2, \quad e_3 \circ e_3 = e_1+\beta e_3
	\end{cases}
	$

	\item \textbf{NP21$_{\alpha,\beta}$:}
	$
	\begin{cases}
		e_1 \cdot e_1 = e_2, \quad e_1 \cdot e_3 = e_1, \quad e_2 \cdot e_3 = e_2, \quad e_3 \cdot e_3 = e_3 \\
		e_1 \circ e_1 = (\frac{\alpha+\beta}{2}) e_2, \quad e_1 \circ e_3 = \beta e_1, \quad e_2 \circ e_3 = \beta e_2 \\
		e_3 \circ e_1 = (\frac{\alpha+\beta}{2}) e_1, \quad e_3 \circ e_2 = \alpha e_2, \quad e_3 \circ e_3 = e_2+\beta e_3
	\end{cases}
	$

	\item \textbf{NP22$_{\alpha,\beta}^{\beta\neq 0}$:}
	$
	\begin{cases}
		e_1 \cdot e_1 = e_2, \quad e_1 \cdot e_3 = e_1, \quad e_2 \cdot e_3 = e_2, \quad e_3 \cdot e_3 = e_3 \\
		e_1 \circ e_1 = \alpha e_2, \quad e_1 \circ e_3 = \alpha e_1, \quad e_2 \circ e_3 = \alpha e_2, \quad e_3 \circ e_1 = \alpha e_1+\beta e_2 \\
		e_3 \circ e_2 = \alpha e_2, \quad e_3 \circ e_3 = e_2+\alpha e_3
	\end{cases}
	$

	\item \textbf{NP23$_{\alpha,\beta}$:}
	$
	\begin{cases}
		e_1 \cdot e_1 = e_2, \quad e_1 \cdot e_3 = e_1, \quad e_2 \cdot e_3 = e_2, \quad e_3 \cdot e_3 = e_3 \\
		e_1 \circ e_1 = (\frac{\alpha+\beta}{2}) e_2, \quad e_1 \circ e_3 = \alpha e_1, \quad e_2 \circ e_3 = \alpha e_2 \\
		e_3 \circ e_1 = (\frac{\alpha+\beta}{2}) e_1, \quad e_3 \circ e_2 = \beta e_2, \quad e_3 \circ e_3 = \alpha e_3
	\end{cases}
	$

	\item \textbf{NP24$_{\alpha}$:}
	$
	\begin{cases}
		e_1 \cdot e_1 = e_2, \quad e_1 \cdot e_3 = e_1, \quad e_2 \cdot e_3 = e_2, \quad e_3 \cdot e_3 = e_3 \\
		e_1 \circ e_1 = \alpha e_2, \quad e_1 \circ e_3 = \alpha e_1, \quad e_2 \circ e_3 = \alpha e_2, \quad e_3 \circ e_1 = \alpha e_1+e_2 \\
		e_3 \circ e_2 = \alpha e_2, \quad e_3 \circ e_3 = \alpha e_3
	\end{cases}
	$

\item \textbf{NP25$_{\alpha,\beta,\gamma}$:}
	$
	\begin{cases}
		e_1 \cdot e_1 = e_1, \quad e_2 \cdot e_2 = e_2, \quad e_3 \cdot e_3 = e_3 \\
		e_1 \circ e_1 = \alpha e_1, \quad e_2 \circ e_2 = \beta e_2, \quad e_3 \circ e_3 = \gamma e_3
	\end{cases}
	$

	\item \textbf{NP26$_{\alpha,\beta,\gamma}$:}
	$
	\begin{cases}
		e_1 \cdot e_1 = e_1, \quad e_2 \cdot e_2 = e_2, \quad e_1 \cdot e_3 = e_3 \\
		e_1 \circ e_1 = \alpha e_1 + e_3, \quad e_1 \circ e_3 = \beta e_3, \quad e_2 \circ e_2 = \gamma e_2, \quad e_3 \circ e_1 = \alpha e_3
	\end{cases}
	$

	\item \textbf{NP27$_{\alpha,\beta,\gamma}$:}
	$
	\begin{cases}
		e_1 \cdot e_1 = e_1, \quad e_2 \cdot e_2 = e_2, \quad e_1 \cdot e_3 = e_3 \\
		e_1 \circ e_1 = \alpha e_1, \quad e_1 \circ e_3 = \beta e_3, \quad e_2 \circ e_2 = \gamma e_2, \quad e_3 \circ e_1 = \alpha e_3
	\end{cases}
	$

	\item \textbf{NP28$_{\alpha,\beta}$:}
	$
	\begin{cases}
		e_1 \cdot e_1 = e_1, \quad e_2 \cdot e_2 = e_2 \\
		e_1 \circ e_1 = \alpha e_1, \quad e_2 \circ e_2 = \beta e_2, \quad e_3 \circ e_3 = e_3
	\end{cases}
	$

	\item \textbf{NP29$_{\alpha,\beta}$:}
	$
	\begin{cases}
		e_1 \cdot e_1 = e_1, \quad e_2 \cdot e_2 = e_2 \\
		e_1 \circ e_1 = \alpha e_1, \quad e_2 \circ e_2 = \beta e_2
	\end{cases}
	$

	\item \textbf{NP30$_{\alpha,\beta}$:}
	$
	\begin{cases}
		e_1 \cdot e_1 = e_1, \quad e_2 \cdot e_2 = e_3 \\
		e_1 \circ e_1 = \alpha e_1, \quad e_2 \circ e_2 = \beta e_2, \quad e_2 \circ e_3 = e_3, \quad e_3 \circ e_2 = \beta e_3
	\end{cases}
	$

	\item \textbf{NP31$_{\alpha,\beta}$:}
	$
	\begin{cases}
		e_1 \cdot e_1 = e_1, \quad e_2 \cdot e_2 = e_3 \\
		e_1 \circ e_1 = \alpha e_1, \quad e_2 \circ e_2 = e_2 + \beta e_3, \quad e_3 \circ e_2 = e_3
	\end{cases}
	$

	\item \textbf{NP32$_{\alpha,\beta}$:}
	$
	\begin{cases}
		e_1 \cdot e_1 = e_1, \quad e_2 \cdot e_2 = e_3 \\
		e_1 \circ e_1 = \alpha e_1, \quad e_2 \circ e_2 = \beta e_3
	\end{cases}
	$

	\item \textbf{NP33$_{\alpha,\beta}$:}
	$
	\begin{cases}
		e_1 \cdot e_1 = e_1 \\
		e_1 \circ e_1 = \alpha e_1, \quad e_2 \circ e_2 = \beta e_2, \quad e_2 \circ e_3 = e_3, \quad e_3 \circ e_2 = \beta e_3
	\end{cases}
	$

\item \textbf{NP34$_{\alpha}$:}
$
\begin{cases}
e_1 \cdot e_1 = e_1 \\
e_1 \circ e_1 = \alpha e_1, \quad e_2 \circ e_2 = e_2+e_3, \quad e_3 \circ e_2 = e_3
\end{cases}
$

\item \textbf{NP35$_{\alpha}$:}
$
\begin{cases}
e_1 \cdot e_1 = e_1 \\
e_1 \circ e_1 = \alpha e_1, \quad e_2 \circ e_2 = e_2, \quad e_3 \circ e_2 = e_3
\end{cases}
$

\item \textbf{NP36$_{\alpha}$:}
$
\begin{cases}
e_1 \cdot e_1 = e_1 \\
e_1 \circ e_1 = \alpha e_1, \quad e_2 \circ e_2 = e_3
\end{cases}
$

\item \textbf{NP37$_{\alpha}$:}
$
\begin{cases}
e_1 \cdot e_1 = e_1 \\
e_1 \circ e_1 = \alpha e_1
\end{cases}
$

\item \textbf{NP38$_{\alpha}$:}
$
\begin{cases}
e_1 \cdot e_1 = e_1 \\
e_1 \circ e_1 = \alpha e_1, \quad e_2 \circ e_2 = e_2, \quad e_3 \circ e_3 = e_3
\end{cases}
$

\item \textbf{NP39$_{\alpha}$:}
$
\begin{cases}
e_1 \cdot e_1 = e_1 \\
e_1 \circ e_1 = \alpha e_1, \quad e_2 \circ e_2 = e_2
\end{cases}
$

	\item \textbf{NP40$_{\alpha,\beta}^{\alpha \neq \beta}$:}
	$
	\begin{cases}
		e_1 \cdot e_1 = e_2 \\
		e_1 \circ e_1 = e_1, \quad e_1 \circ e_2 = \alpha e_2, \quad e_1 \circ e_3 = \beta e_3, \quad e_2 \circ e_1 = e_2, \quad e_3 \circ e_1 = e_3
	\end{cases}
	$

	\item \textbf{NP41$_{\alpha}$:}
	$
	\begin{cases}
		e_1 \cdot e_1 = e_2 \\
		e_1 \circ e_1 = e_1, \quad e_1 \circ e_2 = \alpha e_2, \quad e_1 \circ e_3 = e_2 + \alpha e_3, \quad e_2 \circ e_1 = e_2, \quad e_3 \circ e_1 = e_3
	\end{cases}
	$

	\item \textbf{NP42$_{\alpha}$:}
	$
	\begin{cases}
		e_1 \cdot e_1 = e_2 \\
		e_1 \circ e_1 = e_1, \quad e_1 \circ e_2 = \alpha e_2, \quad e_1 \circ e_3 = \alpha e_3, \quad e_2 \circ e_1 = e_2, \quad e_3 \circ e_1 = e_3
	\end{cases}
	$

	\item \textbf{NP43$_{\alpha,\beta}^{\alpha \neq 0}$:}
	$
	\begin{cases}
		e_1 \cdot e_1 = e_2 \\
		e_1 \circ e_1 = e_1 + \alpha e_2, \quad e_1 \circ e_3 = \beta e_3, \quad e_2 \circ e_1 = e_2, \quad e_3 \circ e_1 = e_3
	\end{cases}
	$

	\item \textbf{NP44$_{\alpha,\beta}^{\beta \neq 0}$:}
	$
	\begin{cases}
		e_1 \cdot e_1 = e_2 \\
		e_1 \circ e_1 = e_1, \quad e_1 \circ e_2 = \alpha e_2 + e_3, \quad e_1 \circ e_3 = \beta e_2, \quad e_2 \circ e_1 = e_2, \quad e_3 \circ e_1 = e_3
	\end{cases}
	$

	\item \textbf{NP45$_{\alpha,\beta}$:}
	$
	\begin{cases}
		e_1 \cdot e_1 = e_2 \\
		e_1 \circ e_1 = e_1 + \alpha e_2, \quad e_1 \circ e_2 = \beta e_2 + e_3, \quad e_2 \circ e_1 = e_2, \quad e_3 \circ e_1 = e_3
	\end{cases}
	$

	\item \textbf{NP46$_{\alpha}^{\alpha \neq 0}$:}
	$
	\begin{cases}
		e_1 \cdot e_1 = e_2 \\
		e_1 \circ e_1 = e_1 + e_3, \quad e_1 \circ e_2 = \alpha e_2, \quad e_2 \circ e_1 = e_2, \quad e_3 \circ e_1 = e_3
	\end{cases}
	$

	\item \textbf{NP47$_{\alpha}$:}
	$
	\begin{cases}
		e_1 \cdot e_1 = e_2 \\
		e_1 \circ e_1 = e_1 + e_3, \quad e_1 \circ e_3 = \alpha e_2, \quad e_2 \circ e_1 = e_2, \quad e_3 \circ e_1 = e_3
	\end{cases}
	$

	\item \textbf{NP48$_{\alpha}$:}
	$
	\begin{cases}
		e_1 \cdot e_1 = e_2 \\
		e_1 \circ e_1 = \alpha e_2, \quad e_1 \circ e_2 = e_3, \quad e_1 \circ e_3 = e_3
	\end{cases}
	$

	\item \textbf{NP49$_{\alpha}$:}
	$
	\begin{cases}
		e_1 \cdot e_1 = e_2 \\
		e_1 \circ e_2 = e_2, \quad e_1 \circ e_3 = \alpha e_3
	\end{cases}
	$

	\item \textbf{NP50:}
	$
	\begin{cases}
		e_1 \cdot e_1 = e_2 \\
		e_1 \circ e_2 = e_2, \quad e_1 \circ e_3 = e_2 + e_3
	\end{cases}
	$

	\item \textbf{NP51$_{\alpha}$:}
	$
	\begin{cases}
		e_1 \cdot e_1 = e_2 \\
		e_1 \circ e_1 = \alpha e_2, \quad e_1 \circ e_3 = e_3
	\end{cases}
	$

	\item \textbf{NP52$_{\alpha}$:}
	$
	\begin{cases}
		e_1 \cdot e_1 = e_2 \\
		e_1 \circ e_2 = \alpha e_2 + e_3, \quad e_1 \circ e_3 = e_2
	\end{cases}
	$

	\item \textbf{NP53:}
	$
	\begin{cases}
		e_1 \cdot e_1 = e_2 \\
		e_1 \circ e_1 = e_3, \quad e_1 \circ e_3 = e_2
	\end{cases}
	$

	\item \textbf{NP54:}
	$
	\begin{cases}
		e_1 \cdot e_1 = e_2 \\
		e_1 \circ e_3 = e_2
	\end{cases}
	$

	\item \textbf{NP55$_{\alpha}$:}
	$
	\begin{cases}
		e_1 \cdot e_1 = e_2 \\
		e_1 \circ e_1 = \alpha e_2, \quad e_1 \circ e_2 = e_3
	\end{cases}
	$

	\item \textbf{NP56:}
	$
	\begin{cases}
		e_1 \cdot e_1 = e_2 \\
		e_1 \circ e_1 = e_3, \quad e_1 \circ e_2 = e_2
	\end{cases}
	$

	\item \textbf{NP57:}
	$
	\begin{cases}
		e_1 \cdot e_1 = e_2 \\
		e_1 \circ e_1 = e_3
	\end{cases}
	$

	\item \textbf{NP58$_{\alpha}$:}
	$
	\begin{cases}
		e_1 \cdot e_1 = e_2 \\
		e_1 \circ e_1 = \alpha e_2
	\end{cases}
	$

	\item \textbf{NP59$_{\alpha}$:}
	$
	\begin{cases}
		e_1 \cdot e_1 = e_2 \\
		e_1 \circ e_1 = \alpha e_2, \quad e_1 \circ e_3 = e_1+e_2, \quad e_2 \circ e_3 = e_2, \quad e_3 \circ e_2 = -e_2, \quad e_3 \circ e_3 = e_3
	\end{cases}
	$

	\item \textbf{NP60$_{\alpha}$:}
	$
	\begin{cases}
		e_1 \cdot e_1 = e_2 \\
		e_1 \circ e_1 = \alpha e_2, \quad e_1 \circ e_3 = e_1, \quad e_2 \circ e_3 = e_2, \quad e_3 \circ e_2 = -e_2, \quad e_3 \circ e_3 = e_3
	\end{cases}
	$

	\item \textbf{NP61$_{\alpha}^{\alpha \neq 1,2}$:}
	$
	\begin{cases}
		e_1 \cdot e_1 = e_2 \\
		e_1 \circ e_1 = \alpha e_1, \quad e_1 \circ e_2 = (2-\alpha) e_2, \quad e_1 \circ e_3 = e_3, \quad e_2 \circ e_1 = \alpha e_2,\\ e_3 \circ e_1 = \alpha e_3, \quad e_3 \circ e_3 = e_2
	\end{cases}
	$

	\item \textbf{NP62$_{\alpha}$:}
	$
	\begin{cases}
		e_1 \cdot e_1 = e_2 \\
		e_1 \circ e_1 = e_1, \quad e_1 \circ e_2 = e_2, \quad e_1 \circ e_3 = \alpha e_2+e_3, \quad e_2 \circ e_1 = e_2, \\ e_3 \circ e_1 = e_3, \quad e_3 \circ e_3 = e_2
	\end{cases}
	$

	\item \textbf{NP63$_{\alpha}$:}
	$
	\begin{cases}
		e_1 \cdot e_1 = e_2 \\
		e_1 \circ e_1 = 2e_1+\alpha e_2, \quad e_1 \circ e_3 = e_3, \quad e_2 \circ e_1 = 2e_2, \quad e_3 \circ e_1 = 2e_3, \quad e_3 \circ e_3 = e_2
	\end{cases}
	$

	\item \textbf{NP64:}
	$
	\begin{cases}
		e_1 \cdot e_1 = e_2 \\
		e_1 \circ e_1 = e_1, \quad e_1 \circ e_2 = -e_2, \quad e_2 \circ e_1 = e_2, \quad e_3 \circ e_1 = e_3, \quad e_3 \circ e_3 = e_2
	\end{cases}
	$

	\item \textbf{NP65$_{\alpha,\beta}$:}
	$
	\begin{cases}
		e_1 \cdot e_1 = e_2 \\
		e_1 \circ e_1 = \alpha e_2, \quad e_1 \circ e_3 = \beta e_2, \quad e_3 \circ e_3 = e_2
	\end{cases}
	$

	\item \textbf{NP66$_{\alpha}$:}
	$
	\begin{cases}
		e_1 \cdot e_1 = e_2 \\
		e_1 \circ e_1 = \alpha e_1, \quad e_1 \circ e_2 = e_2, \quad e_2 \circ e_1 = \alpha e_2, \quad e_3 \circ e_3 = e_3
	\end{cases}
	$

	\item \textbf{NP67$_{\alpha}$:}
	$
	\begin{cases}
		e_1 \cdot e_1 = e_2 \\
		e_1 \circ e_1 = e_1+\alpha e_2, \quad e_2 \circ e_1 = e_2, \quad e_3 \circ e_3 = e_3
	\end{cases}
	$

	\item \textbf{NP68$_{\alpha}$:}
	$
	\begin{cases}
		e_1 \cdot e_1 = e_2 \\
		e_1 \circ e_1 = \alpha e_2, \quad e_3 \circ e_3 = e_3
	\end{cases}
	$

	\item \textbf{NP69$_{\alpha}$:}
	$
	\begin{cases}
		e_1 \cdot e_1 = e_2 \\
		e_1 \circ e_1 = e_1+\alpha e_3, \quad e_1 \circ e_2 = -e_2, \quad e_2 \circ e_1 = e_2, \quad e_3 \circ e_1 = e_2+e_3
	\end{cases}
	$

	\item \textbf{NP70$_{\alpha}$:}
	$
	\begin{cases}
		e_1 \cdot e_1 = e_2 \\
		e_1 \circ e_1 = e_3, \quad e_1 \circ e_3 = \alpha e_2, \quad e_3 \circ e_1 = e_2
	\end{cases}
	$

	\item \textbf{NP71$_{\alpha}^{\alpha \neq -1}$:}
	$
	\begin{cases}
		e_1 \cdot e_1 = e_2 \\
		e_1 \circ e_3 = \alpha e_2, \quad e_3 \circ e_1 = e_2
	\end{cases}
	$

	\item \textbf{NP72$_{\alpha}$:}
	$
	\begin{cases}
		e_1 \cdot e_1 = e_2 \\
		e_1 \circ e_1 = \alpha e_2, \quad e_1 \circ e_3 = -e_2, \quad e_3 \circ e_1 = e_2
	\end{cases}
	$

	\item \textbf{NP73$_{\alpha}^{\alpha \neq 0}$:}
	$
	\begin{cases}
		e_1 \cdot e_1 = e_2 \\
		e_1 \circ e_1 = e_1, \quad e_1 \circ e_2 = \alpha e_2, \quad e_2 \circ e_1 = e_2
	\end{cases}
	$

	\item \textbf{NP74$_{\alpha}$:}
	$
	\begin{cases}
		e_1 \cdot e_1 = e_2 \\
		e_1 \circ e_1 = e_1+\alpha e_2, \quad e_2 \circ e_1 = e_2
	\end{cases}
	$
\\    
$\mathbf{NP01}^{(\alpha_1, \beta_1)} \cong \mathbf{NP01}^{(\alpha_2, \beta_2)}$ if and only if $(\alpha_1, \beta_1) = (\alpha_2, \beta_2)$ or $(\alpha_1, \beta_1) = (-\beta_2, -\alpha_2)$. \\
$\mathbf{NP02}^{\alpha_1} \cong \mathbf{NP02}^{\alpha_2}$ if and only if $\alpha_1 = \alpha_2$. \\
$\mathbf{NP03}^{\alpha_1} \cong \mathbf{NP03}^{\alpha_2}$ if and only if $\alpha_1 = \alpha_2$. \\
$\mathbf{NP04}^{(\alpha_1, \beta_1, 0)} \cong \mathbf{NP04}^{(\alpha_2, \beta_2, 0)}$ if and only if $(\alpha_1, \beta_1) = (\alpha_2, \beta_2)$. \\
$\mathbf{NP04}^{(\alpha_1, \beta_1, \gamma_1\neq 0)} \cong \mathbf{NP04}^{(\alpha_2, \beta_2, \gamma_2\neq 0)}$ if and only if $(\alpha_1, \beta_1,\gamma_1) = (\alpha_2, \beta_2,\gamma_2)$ or $(\alpha_1, \beta_1,\gamma_1) = (\beta_2, \alpha_2,\gamma_2)$. \\
$\mathbf{NP05}^{(\alpha_1, \beta_1)} \cong \mathbf{NP05}^{(\alpha_2, \beta_2)}$ if and only if $(\alpha_1, \beta_1) = (\alpha_2, \beta_2)$ or $(\alpha_1, \beta_1) = ( \beta_2, \alpha_2)$. \\
$\mathbf{NP06}^{(\alpha_1, \beta_1)} \cong \mathbf{NP06}^{(\alpha_2, \beta_2)}$ if and only if $(\alpha_1, \beta_1) = (\alpha_2, \beta_2)$ ($\alpha \neq 2$). \\
$\mathbf{NP07}^{(\alpha_1, \beta_1)} \cong \mathbf{NP07}^{(\alpha_2, \beta_2)}$ if and only if $(\alpha_1, \beta_1) = (\alpha_2, \beta_2)$. \\
$\mathbf{NP08}^{\alpha_1} \cong \mathbf{NP08}^{\alpha_2}$ if and only if $\alpha_1 = \alpha_2$. \\
$\mathbf{NP09}^{(\alpha_1, \beta_1)} \cong \mathbf{NP09}^{(\alpha_2, \beta_2)}$ if and only if $(\alpha_1, \beta_1) = (\alpha_2, \beta_2)$. \\
$\mathbf{NP10}^{\alpha_1} \cong \mathbf{NP10}^{\alpha_2}$ if and only if $\alpha_1 = \alpha_2$. \\
$\mathbf{NP11}^{(\alpha_1, \beta_1, \gamma_1)} \cong \mathbf{NP11}^{(\alpha_2, \beta_2, \gamma_2)}$ if and only if $(\alpha_1, \beta_1, \gamma_1) = (\alpha_2, \beta_2, \gamma_2)$. \\
$\mathbf{NP12}^{(\alpha_1, \beta_1, \gamma_1)} \cong \mathbf{NP12}^{(\alpha_2, \beta_2, \gamma_2)}$ if and only if $(\alpha_1, \beta_1, \gamma_1) = (\alpha_2, \beta_2, \gamma_2)$. \\
$\mathbf{NP13}^{(\alpha_1, \beta_1)} \cong \mathbf{NP13}^{(\alpha_2, \beta_2)}$ if and only if $(\alpha_1, \beta_1) = (\alpha_2, \beta_2)$. \\
$\mathbf{NP14}^{(\alpha_1, \beta_1, \gamma_1)} \cong \mathbf{NP14}^{(\alpha_2, \beta_2, \gamma_2)}$ if and only if $(\alpha_1, \beta_1, \gamma_1) = (\alpha_2, \beta_2, \gamma_2)$. \\
$\mathbf{NP15}^{(\alpha_1, \beta_1)} \cong \mathbf{NP15}^{(\alpha_2, \beta_2)}$ if and only if $(\alpha_1, \beta_1) = (\alpha_2, \beta_2)$. \\
$\mathbf{NP16}^{(\alpha_1, \beta_1)} \cong \mathbf{NP16}^{(\alpha_2, \beta_2)}$ if and only if $(\alpha_1, \beta_1) = (\alpha_2, \beta_2)$. \\
$\mathbf{NP17}^{(\alpha_1, \beta_1)} \cong \mathbf{NP17}^{(\alpha_2, \beta_2)}$ if and only if $(\alpha_1, \beta_1) = (\alpha_2, \beta_2)$. \\
$\mathbf{NP18}^{(\alpha_1, \beta_1)} \cong \mathbf{NP18}^{(\alpha_2, \beta_2)}$ if and only if $(\alpha_1, \beta_1) = (\alpha_2, \beta_2)$. \\
$\mathbf{NP19}^{(\alpha_1, \beta_1)} \cong \mathbf{NP19}^{(\alpha_2, \beta_2)}$ if and only if $(\alpha_1, \beta_1) = (\alpha_2, \beta_2)$. \\
$\mathbf{NP20}^{(\alpha_1, \beta_1,\gamma_1)} \cong \mathbf{NP20}^{(\alpha_2, \beta_2,\gamma_2)}$ if and only if $(\alpha_1, \beta_1,\gamma_1) = (\alpha_2, \beta_2,\gamma_2)$. \\
$\mathbf{NP21}^{(\alpha_1, \beta_1)} \cong \mathbf{NP21}^{(\alpha_2, \beta_2)}$ if and only if $(\alpha_1, \beta_1) = (\alpha_2, \beta_2)$. \\
$\mathbf{NP22}^{(\alpha_1, \beta_1)} \cong \mathbf{NP22}^{(\alpha_2, \beta_2)}$ if and only if $(\alpha_1, \beta_1) = (\alpha_2, \pm \beta_2)$ ($\beta \neq 0$). \\
$\mathbf{NP23}^{(\alpha_1, \beta_1)} \cong \mathbf{NP23}^{(\alpha_2, \beta_2)}$ if and only if $(\alpha_1, \beta_1) = (\alpha_2, \beta_2)$. \\
$\mathbf{NP24}^{\alpha_1} \cong \mathbf{NP24}^{\alpha_2}$ if and only if $\alpha_1 = \alpha_2$. \\
$\mathbf{NP25}^{(\alpha_1, \beta_1, \gamma_1)} \cong \mathbf{NP25}^{(\alpha_2, \beta_2, \gamma_2)}$ if and only if $(\alpha_1, \beta_1, \gamma_1)$ is a permutation of $(\alpha_2, \beta_2, \gamma_2)$. \\
$\mathbf{NP26}^{(\alpha_1, \beta_1, \gamma_1)} \cong \mathbf{NP26}^{(\alpha_2, \beta_2, \gamma_2)}$ if and only if $(\alpha_1, \beta_1, \gamma_1) = (\alpha_2, \beta_2, \gamma_2)$. \\
$\mathbf{NP27}^{(\alpha_1, \beta_1, \gamma_1)} \cong \mathbf{NP27}^{(\alpha_2, \beta_2, \gamma_2)}$ if and only if $(\alpha_1, \beta_1, \gamma_1) = (\alpha_2, \beta_2, \gamma_2)$. \\
$\mathbf{NP28}^{(\alpha_1, \beta_1)} \cong \mathbf{NP28}^{(\alpha_2, \beta_2)}$ if and only if $(\alpha_1, \beta_1) = (\alpha_2, \beta_2)$ or $(\alpha_1, \beta_1) = ( \beta_2,\alpha_2)$.\\
$\mathbf{NP29}^{(\alpha_1, \beta_1)} \cong \mathbf{NP29}^{(\alpha_2, \beta_2)}$ if and only if $(\alpha_1, \beta_1) = (\alpha_2, \beta_2)$ or $(\alpha_1, \beta_1) = (\beta_2,\alpha_2)$. \\
$\mathbf{NP30}^{(\alpha_1, \beta_1)} \cong \mathbf{NP30}^{(\alpha_2, \beta_2)}$ if and only if $(\alpha_1, \beta_1) = ( \alpha_2, \beta_2)$. \\
$\mathbf{NP31}^{(\alpha_1, \beta_1)} \cong \mathbf{NP31}^{(\alpha_2, \beta_2)}$ if and only if $(\alpha_1, \beta_1) = (\alpha_2, \beta_2)$. \\
$\mathbf{NP32}^{(\alpha_1, \beta_1)} \cong \mathbf{NP32}^{(\alpha_2, \beta_2)}$ if and only if $(\alpha_1, \beta_1) = (\alpha_2, \beta_2)$. \\
$\mathbf{NP33}^{(\alpha_1,\beta_1)} \cong \mathbf{NP33}^{(\alpha_2,\beta_2)}$ 
if and only if $(\alpha_1,\beta_1)=(\alpha_2,\beta_2)$. \\
$\mathbf{NP34}^{\alpha_1} \cong \mathbf{NP34}^{\alpha_2}$ if and only if $\alpha_1 = \alpha_2$. \\
$\mathbf{NP35}^{\alpha_1} \cong \mathbf{NP35}^{\alpha_2}$ if and only if $\alpha_1 = \alpha_2$. \\
$\mathbf{NP36}^{\alpha_1} \cong \mathbf{NP36}^{\alpha_2}$ if and only if $\alpha_1 = \alpha_2$. \\
$\mathbf{NP37}^{\alpha_1} \cong \mathbf{NP37}^{\alpha_2}$ if and only if $\alpha_1 = \alpha_2$. \\
$\mathbf{NP38}^{\alpha_1} \cong \mathbf{NP38}^{\alpha_2}$ if and only if $\alpha_1 = \alpha_2$. \\
$\mathbf{NP39}^{\alpha_1} \cong \mathbf{NP39}^{\alpha_2}$ if and only if $\alpha_1 = \alpha_2$. \\
$\mathbf{NP40}^{(\alpha_1, \beta_1)} \cong \mathbf{NP40}^{(\alpha_2, \beta_2)}$ if and only if $(\alpha_1, \beta_1) = (\alpha_2, \beta_2)$ ($\alpha \neq \beta$). \\
$\mathbf{NP41}^{\alpha_1} \cong \mathbf{NP41}^{\alpha_2}$ if and only if $\alpha_1 = \alpha_2$. \\
$\mathbf{NP42}^{\alpha_1} \cong \mathbf{NP42}^{\alpha_2}$ if and only if $\alpha_1 = \alpha_2$. \\
$\mathbf{NP43}^{(\alpha_1, \beta_1)} \cong \mathbf{NP43}^{(\alpha_2, \beta_2)}$ if and only if $(\alpha_1, \beta_1) = (\alpha_2, \beta_2)$ ($\alpha \neq 0$). \\
$\mathbf{NP44}^{(\alpha_1, \beta_1)} \cong \mathbf{NP44}^{(\alpha_2, \beta_2)}$ if and only if $(\alpha_1, \beta_1) = (\alpha_2, \beta_2)$ ($\beta \neq 0$). \\
$\mathbf{NP45}^{(\alpha_1, \beta_1)} \cong \mathbf{NP45}^{(\alpha_2, \beta_2)}$ if and only if $(\alpha_1, \beta_1) = (\alpha_2, \beta_2)$. \\
$\mathbf{NP46}^{\alpha_1} \cong \mathbf{NP46}^{\alpha_2}$ if and only if $\alpha_1 = \alpha_2$ ($\alpha \neq 0$). \\
$\mathbf{NP47}^{\alpha_1} \cong \mathbf{NP47}^{\alpha_2}$ if and only if $\alpha_1 = \alpha_2$. \\
$\mathbf{NP48}^{\alpha_1} \cong \mathbf{NP48}^{\alpha_2}$ if and only if $\alpha_1 = \alpha_2$. \\
$\mathbf{NP49}^{\alpha_1} \cong \mathbf{NP49}^{\alpha_2}$ if and only if $\alpha_1 = \alpha_2$. \\
$\mathbf{NP51}^{\alpha_1} \cong \mathbf{NP51}^{\alpha_2}$ if and only if $\alpha_1 = \alpha_2$. \\
$\mathbf{NP52}^{\alpha_1} \cong \mathbf{NP52}^{\alpha_2}$ if and only if $\alpha_1 =\pm \alpha_2$. \\
$\mathbf{NP55}^{\alpha_1} \cong \mathbf{NP55}^{\alpha_2}$ if and only if $\alpha_1 = \alpha_2$. \\
$\mathbf{NP58}^{\alpha_1} \cong \mathbf{NP58}^{\alpha_2}$ if and only if $\alpha_1 = \alpha_2$. \\
$\mathbf{NP59}^{\alpha_1} \cong \mathbf{NP59}^{\alpha_2}$ if and only if $\alpha_1 = \alpha_2$. \\
$\mathbf{NP60}^{\alpha_1} \cong \mathbf{NP60}^{\alpha_2}$ if and only if $\alpha_1 = \alpha_2$. \\
$\mathbf{NP61}^{\alpha_1} \cong \mathbf{NP61}^{\alpha_2}$ if and only if $\alpha_1 = \alpha_2$ ($\alpha \neq 1,2$). \\
$\mathbf{NP62}^{\alpha_1} \cong \mathbf{NP62}^{\alpha_2}$ if and only if $\alpha_1 = \pm \alpha_2$. \\
$\mathbf{NP63}^{\alpha_1} \cong \mathbf{NP63}^{\alpha_2}$ if and only if $\alpha_1 = \alpha_2$. \\
$\mathbf{NP65}^{(\alpha_1, \beta_1)} \cong \mathbf{NP65}^{(\alpha_2, \beta_2)}$ if and only if $(\alpha_1, \beta_1) = (\alpha_2, \pm \beta_2)$. \\
$\mathbf{NP66}^{\alpha_1} \cong \mathbf{NP66}^{\alpha_2}$ if and only if $\alpha_1 = \alpha_2$. \\
$\mathbf{NP67}^{\alpha_1} \cong \mathbf{NP67}^{\alpha_2}$ if and only if $\alpha_1 = \alpha_2$. \\
$\mathbf{NP68}^{\alpha_1} \cong \mathbf{NP68}^{\alpha_2}$ if and only if $\alpha_1 = \alpha_2$. \\
$\mathbf{NP69}^{\alpha_1} \cong \mathbf{NP69}^{\alpha_2}$ if and only if $\alpha_1 = \alpha_2$. \\
$\mathbf{NP70}^{\alpha_1} \cong \mathbf{NP70}^{\alpha_2}$ if and only if $\alpha_1 = \alpha_2$. \\
$\mathbf{NP71}^{\alpha_1} \cong \mathbf{NP71}^{\alpha_2}$ if and only if $\alpha_1 = \alpha_2$ ($\alpha \neq -1$). \\
$\mathbf{NP72}^{\alpha_1} \cong \mathbf{NP72}^{\alpha_2}$ if and only if $\alpha_1 = \alpha_2$. \\
$\mathbf{NP73}^{\alpha_1} \cong \mathbf{NP73}^{\alpha_2}$ if and only if $\alpha_1 = \alpha_2$ ($\alpha \neq 0$). \\
$\mathbf{NP74}^{\alpha_1} \cong \mathbf{NP74}^{\alpha_2}$ if and only if $\alpha_1 = \alpha_2$.

\end{enumerate}

\end{thm}

\begin{proof}
We give the complete argument for the commutative associative algebra \(\mathrm{A}_{1,7}\). The classifications of transposed Poisson algebras over the remaining commutative associative algebras are similar. According to Theorem~\ref{thm:class}, its nonzero products are
\[
e_1\cdot e_1=e_1,\qquad e_2\cdot e_2=e_2,\qquad
e_1\cdot e_3=e_3.
\]
We use the classification method described in Section~\ref{sec:method}. We first solve the polynomial system for all Novikov--Poisson products, then perform isomorphism reduction, and finally carry out isomorphism testing.

By Remark~\ref{rem:first}, all Novikov--Poisson products on $(\mathrm A_{1,7},\cdot)$ are obtained by solving
\eqref{eq:NP-polynomial-system}. Let
\[
L_{e_1}=
\begin{pmatrix}
a_{11}&a_{12}&a_{13}\\
a_{21}&a_{22}&a_{23}\\
a_{31}&a_{32}&a_{33}
\end{pmatrix},\quad
L_{e_2}=
\begin{pmatrix}
b_{11}&b_{12}&b_{13}\\
b_{21}&b_{22}&b_{23}\\
b_{31}&b_{32}&b_{33}
\end{pmatrix},\quad
L_{e_3}=
\begin{pmatrix}
c_{11}&c_{12}&c_{13}\\
c_{21}&c_{22}&c_{23}\\
c_{31}&c_{32}&c_{33}
\end{pmatrix}
\]
be the left multiplication matrices of a bilinear product $\circ$, where $L_{e_i}(e_j)=e_i\circ e_j$.

The third family of equations in \eqref{eq:NP-polynomial-system}, which is the matrix form of the compatibility identity \eqref{eq:NP3}, gives
\[
a_{21}=a_{22}=a_{23}=0,\qquad
b_{11}=b_{12}=b_{13}=b_{31}=b_{32}=b_{33}=0,
\]
and
\[
c_{11}=c_{12}=c_{13}=c_{21}=c_{22}=c_{23}=0,\qquad
c_{31}=a_{11},\quad c_{32}=a_{12},\quad c_{33}=a_{13}.
\]
Consequently,
\[
L_{e_1}=
\begin{pmatrix}
a_{11}&a_{12}&a_{13}\\
0&0&0\\
a_{31}&a_{32}&a_{33}
\end{pmatrix},\quad
L_{e_2}=
\begin{pmatrix}
0&0&0\\
b_{21}&b_{22}&b_{23}\\
0&0&0
\end{pmatrix},\quad
L_{e_3}=
\begin{pmatrix}
0&0&0\\
0&0&0\\
a_{11}&a_{12}&a_{13}
\end{pmatrix}.
\]

We next substitute these matrices into the fourth family of equations in \eqref{eq:NP-polynomial-system}, which is the matrix form of \eqref{eq:NP4}. For $(i,j,k)=(1,2,1),(2,3,2),(1,3,3)$, respectively, we obtain
\[
a_{12}e_1+a_{32}e_3=-b_{21}e_2,\qquad
b_{23}e_2=-a_{12}e_3,\qquad
a_{13}e_3=-a_{13}e_3.
\]
Since $\{e_1,e_2,e_3\}$ is linearly independent and the ground field is $\mathbb C$, these equations imply
\[
a_{12}=a_{13}=a_{32}=b_{21}=b_{23}=0.
\]
The remaining equations in the fourth family are then satisfied
identically. Hence
\[
L_{e_1}^{(1)}=
\begin{pmatrix}
c_{31}&0&0\\
0&0&0\\
a_{31}&0&a_{33}
\end{pmatrix},\quad
L_{e_2}^{(1)}=
\begin{pmatrix}
0&0&0\\
0&b_{22}&0\\
0&0&0
\end{pmatrix},\quad
L_{e_3}^{(1)}=
\begin{pmatrix}
0&0&0\\
0&0&0\\
c_{31}&0&0
\end{pmatrix}.
\]

Substituting these matrices into the first two families of equations in \eqref{eq:NP-polynomial-system}, corresponding to the Novikov identities \eqref{eq:novikov-left-symmetric} and \eqref{eq:novikov-right-commutative}, shows that every remaining coordinate polynomial vanishes identically. Thus the complete set of
Novikov--Poisson products on $\mathrm A_{1,7}$ is
\[
e_1\circ e_1=c_{31}e_1+a_{31}e_3,\qquad
e_1\circ e_3=a_{33}e_3,\qquad
e_2\circ e_2=b_{22}e_2,\qquad
e_3\circ e_1=c_{31}e_3,
\]
where $a_{31},a_{33},b_{22},c_{31}\in\mathbb C$ are arbitrary and all omitted products are zero.

We now perform isomorphism reduction in the sense described above. By Proposition~\ref{prop:orbit-criterion}, two products obtained above define isomorphic Novikov--Poisson algebras if and only if they belong to the same
$\operatorname{Aut}(A,\cdot)$-orbit. Let
\[
\varphi=
\begin{pmatrix}
x_{11}&x_{12}&x_{13}\\
x_{21}&x_{22}&x_{23}\\
x_{31}&x_{32}&x_{33}
\end{pmatrix}
\]
be an automorphism of $(A,\cdot)$. Solving
\[
\varphi(e_i\cdot e_j)=\varphi(e_i)\cdot\varphi(e_j),
\qquad i,j\in\{1,2,3\},\qquad \det\varphi\neq0,
\]
gives
\[
x_{11}=x_{22}=1,\qquad
x_{12}=x_{13}=x_{21}=x_{23}=x_{31}=x_{32}=0,\qquad
x_{33}\neq0.
\]
Therefore,
\[
\operatorname{Aut}(A,\cdot)=
\left\{
\varphi_{x_{33}}=
\begin{pmatrix}
1&0&0\\
0&1&0\\
0&0&x_{33}
\end{pmatrix}
\ \middle|\ 
x_{33}\in\mathbb C^*
\right\}.
\]

Fix the parameter tuple
\[
p=(a_{31},a_{33},b_{22},c_{31})
\]
of the structure defined by $L_{e_i}^{(1)}$, and let
\[
q=(aa_{31},aa_{33},bb_{22},cc_{31})
\]
be the parameter tuple of another member of the same family, with matrices
\[
\widetilde L_{e_1}=
\begin{pmatrix}
cc_{31}&0&0\\
0&0&0\\
aa_{31}&0&aa_{33}
\end{pmatrix},\quad
\widetilde L_{e_2}=
\begin{pmatrix}
0&0&0\\
0&bb_{22}&0\\
0&0&0
\end{pmatrix},\quad
\widetilde L_{e_3}=
\begin{pmatrix}
0&0&0\\
0&0&0\\
cc_{31}&0&0
\end{pmatrix}.
\]
In accordance with the isomorphism reduction method, we regard $p$ as fixed and treat $(\varphi,q)$ as the unknowns. By Remark~\ref{rem:second}, the orbit condition in Proposition~\ref{prop:orbit-criterion} is equivalent to \eqref{eq:NP-isomorphism-matrix}. An isomorphism from the structure defined by $\widetilde L_{e_i}$ to the one defined by $L_{e_i}^{(1)}$ therefore satisfies
\[
\left(\sum_{r=1}^{3}x_{ri}L_{e_r}^{(1)}\right)\varphi
=
\varphi\widetilde L_{e_i},
\qquad i=1,2,3.
\]
Substituting
\[
\varphi=\varphi_{x_{33}}
=
\begin{pmatrix}
1&0&0\\
0&1&0\\
0&0&x_{33}
\end{pmatrix}
\]
into these equations gives
\[
\begin{aligned}
&c_{31}-cc_{31}=0,\qquad
a_{31}-x_{33}aa_{31}=0,\qquad
x_{33}(a_{33}-aa_{33})=0,\\
&b_{22}-bb_{22}=0,\qquad
x_{33}(c_{31}-cc_{31})=0,\qquad x_{33}\neq0.
\end{aligned}
\]
To express the nonzero condition by a polynomial equation, introduce an auxiliary variable $D_1$ and add
\[
D_1x_{33}+1=0.
\]
Solving this system for $x_{33},D_1$ and the entries of $q$ gives
\[
aa_{31}=\frac{a_{31}}{x_{33}},\qquad
aa_{33}=a_{33},\qquad
bb_{22}=b_{22},\qquad
cc_{31}=c_{31},\qquad
D_1=-\frac{1}{x_{33}}.
\]

The relation
\[
aa_{31}=\frac{a_{31}}{x_{33}}
\]
is case (ii) of the isomorphism reduction method, since the new parameter $aa_{31}$ depends on both the fixed parameter $a_{31}$ and the automorphism. By contrast,
\[
aa_{33}=a_{33},\qquad bb_{22}=b_{22},\qquad cc_{31}=c_{31}
\]
are case (iii), so these three parameters cannot be reduced further within the family.

If $a_{31}\neq0$, we solve $aa_{31}=1$ by taking $x_{33}=a_{31}$. The corresponding automorphism is
\[
\varphi_1=
\begin{pmatrix}
1&0&0\\
0&1&0\\
0&0&a_{31}
\end{pmatrix}.
\]
It gives an isomorphism from the normalized structure to the original structure, and its inverse transforms the original structure into the normal form
\[
\widetilde L_{e_1}=
\begin{pmatrix}
c_{31}&0&0\\
0&0&0\\
1&0&a_{33}
\end{pmatrix},\quad
\widetilde L_{e_2}=
\begin{pmatrix}
0&0&0\\
0&b_{22}&0\\
0&0&0
\end{pmatrix},\quad
\widetilde L_{e_3}=
\begin{pmatrix}
0&0&0\\
0&0&0\\
c_{31}&0&0
\end{pmatrix}.
\]
Hence every structure with $a_{31}\neq0$ is isomorphic to one in which $a_{31}=1$.

If $a_{31}=0$, then
\[
aa_{31}=\frac{a_{31}}{x_{33}}=0
\]
for every $x_{33}\in\mathbb C^*$. Thus the zero value cannot be changed by an automorphism, and the corresponding normal form is
\[
\widetilde L_{e_1}=
\begin{pmatrix}
c_{31}&0&0\\
0&0&0\\
0&0&a_{33}
\end{pmatrix},\quad
\widetilde L_{e_2}=
\begin{pmatrix}
0&0&0\\
0&b_{22}&0\\
0&0&0
\end{pmatrix},\quad
\widetilde L_{e_3}=
\begin{pmatrix}
0&0&0\\
0&0&0\\
c_{31}&0&0
\end{pmatrix}.
\]
This completes the isomorphism reduction and gives two normal-form families.

We next perform isomorphism testing in the sense described above. By Proposition~\ref{prop:orbit-criterion} and
Remark~\ref{rem:second}, every comparison is determined by \eqref{eq:NP-isomorphism-matrix}. We first fix two parameter tuples in the same normal-form family and solve the system with $\varphi\in\operatorname{Aut}(A,\cdot)$ as the unknown.

Put
\[
\alpha=c_{31},\qquad \beta=a_{33},\qquad \gamma=b_{22}.
\]
For two normalized structures in the nonzero branch with parameter tuples $(\alpha_1,\beta_1,\gamma_1)$ and $(\alpha_2,\beta_2,\gamma_2)$, the isomorphism equations give
\[
\alpha_2=\alpha_1,\qquad
\beta_2=\beta_1,\qquad
\gamma_2=\gamma_1,\qquad
1-x_{33}=0.
\]
Thus $x_{33}=1$, and two structures in this branch are isomorphic if and only if
\[
(\alpha_1,\beta_1,\gamma_1)
=
(\alpha_2,\beta_2,\gamma_2).
\]

For two structures in the zero branch, the isomorphism equations give
\[
\alpha_2=\alpha_1,\qquad
\beta_2=\beta_1,\qquad
\gamma_2=\gamma_1,
\]
while $x_{33}\in\mathbb C^*$ remains arbitrary. Hence two structures in the zero branch are isomorphic if and only if
\[
(\alpha_1,\beta_1,\gamma_1)
=
(\alpha_2,\beta_2,\gamma_2).
\]

It remains to perform the isomorphism testing between the two normal-form families. To determine whether they contain any isomorphic subfamilies, we regard the entries of $\varphi$ and both parameter tuples as unknowns in \eqref{eq:NP-isomorphism-matrix}. Taking a member of the zero branch as the source and a member of the nonzero branch as the target, the equation
\[
a_{31}-x_{33}aa_{31}=0
\]
becomes
\[
1-x_{33}\cdot0=0,
\]
which is impossible. This contradiction is independent of $\alpha,\beta,\gamma$, so the two families contain no isomorphic subfamilies.

Consequently, the isomorphism reduction gives the families
\[
\mathrm{NP26}_{\alpha,\beta,\gamma}:
\quad
\left\{
\begin{aligned}
&e_1\cdot e_1=e_1,\qquad
e_2\cdot e_2=e_2,\qquad
e_1\cdot e_3=e_3,\\
&e_1\circ e_1=\alpha e_1+e_3,\qquad
e_1\circ e_3=\beta e_3,\\
&e_2\circ e_2=\gamma e_2,\qquad
e_3\circ e_1=\alpha e_3,
\end{aligned}
\right.
\]
and
\[
\mathrm{NP27}_{\alpha,\beta,\gamma}:
\quad
\left\{
\begin{aligned}
&e_1\cdot e_1=e_1,\qquad
e_2\cdot e_2=e_2,\qquad
e_1\cdot e_3=e_3,\\
&e_1\circ e_1=\alpha e_1,\qquad
e_1\circ e_3=\beta e_3,\\
&e_2\circ e_2=\gamma e_2,\qquad
e_3\circ e_1=\alpha e_3,
\end{aligned}
\right.
\]
where $\alpha,\beta,\gamma\in\mathbb C$ and all omitted products are zero.

The isomorphism testing shows that
\[
\mathrm{NP26}_{\alpha_1,\beta_1,\gamma_1}
\cong
\mathrm{NP26}_{\alpha_2,\beta_2,\gamma_2}
\]
if and only if
\[
(\alpha_1,\beta_1,\gamma_1)
=
(\alpha_2,\beta_2,\gamma_2),
\]
and
\[
\mathrm{NP27}_{\alpha_1,\beta_1,\gamma_1}
\cong
\mathrm{NP27}_{\alpha_2,\beta_2,\gamma_2}
\]
if and only if
\[
(\alpha_1,\beta_1,\gamma_1)
=
(\alpha_2,\beta_2,\gamma_2).
\]
No member of $\mathrm{NP26}$ is isomorphic to a member of $\mathrm{NP27}$. This completes the classification over
$\mathrm A_{1,7}$.
\end{proof}

\section{Applications of the Classification of 3-Dimensional Novikov--Poisson Algebras}
\label{sec:constructions}

\subsection{3-Dimensional Transposed Poisson Algebras Induced by Novikov--Poisson Algebras}
\label{sec:construction}

We apply the classification obtained in Section~\ref{sec:method} to construct the corresponding transposed Poisson algebras.  We then establish their isomorphism correspondence with the existing classification and study the speciality of 3-dimensional transposed Poisson algebras.

\begin{defn}\cite{Bai}
A \emph{transposed Poisson algebra} is a triple $(A,\cdot,[\,,])$ such that $(A,\cdot)$ is a commutative associative algebra, $(A,[\,,])$ is a Lie algebra, and
\begin{equation}\label{eq:transposed-poisson}
2z\cdot[x,y]=[z\cdot x,y]+[x,z\cdot y]
\end{equation}
for all $x,y,z\in A$.
\end{defn}

\begin{thm}\cite{Bai}\label{thm:commutator-construction}
If $(A,\cdot,\circ)$ is a Novikov--Poisson algebra, then
\begin{equation}\label{eq:commutator}
[x,y]_{\circ}=x\circ y-y\circ x
\end{equation}
defines a transposed Poisson algebra $(A,\cdot,[\,,]_{\circ})$.
\end{thm}

A transposed Poisson algebra obtained from a Novikov--Poisson algebra by taking the commutator of its Novikov product is called special.  Special transposed Poisson algebras in dimension 2 were studied in \cite{BeitesFernandezKaygorodov}. We study the speciality of 3-dimensional transposed Poisson algebras.

\begin{prop}\label{prop:isomorphism-preserved}
If two Novikov--Poisson algebras $(A_1,\cdot_1,\circ_1)$ and $(A_2,\cdot_2,\circ_2)$ are isomorphic, then the transposed Poisson algebras constructed from them by Theorem~\ref{thm:commutator-construction} are also
isomorphic.
\end{prop}
\begin{proof}
The proof is straightforward and therefore omitted.

\end{proof}

For each representative \(\mathrm{NP}ij\) in Theorem~\ref{thm:classification-np3}, let \(\mathrm{TP}ij\) denote the transposed Poisson algebra induced by \eqref{eq:commutator}. The following theorem gives the isomorphism correspondences between these induced algebras and the 3-dimensional transposed Poisson algebras classified in
\cite{Zhang-Z}.

\begin{thm}\label{thm:tp-correspondence}
The isomorphism correspondences between the transposed Poisson algebras \(\mathrm{TP}ij\) induced by Novikov--Poisson algebras and the classes in the existing classification of 3-dimensional transposed Poisson algebras are given in Table~\ref{tab:correspondence}. Here,
$
\mathcal Z=\{T\mid [\, ,\,]_T=0\}
$
denotes the collection of induced algebras with zero Lie bracket, rather than a single isomorphism class.
\end{thm}

\begingroup
\footnotesize
\renewcommand{\arraystretch}{1.06}
\setlength{\tabcolsep}{4pt}
\setlength{\LTleft}{0pt}
\setlength{\LTright}{0pt}

\begin{longtable}{@{}>{\centering\arraybackslash}p{.38\textwidth}
                      p{.58\textwidth}@{}}
\caption{Isomorphism correspondence table.}\label{tab:correspondence}\\
\toprule
{\fontsize{6}{7}\selectfont\bfseries
 Induced Transposed Poisson Algebra}
&
{\fontsize{6}{7}\selectfont\bfseries
 Correspondence with the Existing Classification}\\
\midrule
\endfirsthead

\multicolumn{2}{c}{\tablename~\thetable\ (continued)}\\
\toprule
{\fontsize{6}{7}\selectfont\bfseries
 Induced Transposed Poisson Algebra}
&
{\fontsize{6}{7}\selectfont\bfseries
 Correspondence with the Existing Classification}\\
\midrule
\endhead

\multicolumn{2}{r}{Continued on the next page}\\
\endfoot

\bottomrule
\endlastfoot

\(\mathrm{TP}01_{\alpha,\beta}\)
& \(\mathbf{P16}\)\\
\midrule

\(\mathrm{TP}02_{\alpha},\ \mathrm{TP}03_{\alpha}\)
& \(\mathbf{P18}\)\\
\midrule

\(\mathrm{TP}04_{\alpha,\beta,\gamma},\
  \mathrm{TP}05_{\alpha,\beta}\)
&
\(\begin{cases}
\mathbf{P04}^{1/(\alpha-\beta)}, & \alpha\neq\beta,\\
\mathcal Z, & \alpha=\beta.
\end{cases}\)\\
\midrule

\(\mathrm{TP}06_{\alpha,\beta}^{\,(\alpha\neq2)}\)
&
\(\begin{cases}
\mathbf{P03}^{-1/\beta}, & \alpha=1,\ \beta\neq0,\\
\mathcal Z, & \alpha=1,\ \beta=0,\\
\mathbf{P17}, & \alpha\neq1,\ 2.
\end{cases}\)\\
\midrule

\(\mathrm{TP}07_{\alpha,\beta},\ \mathrm{TP}08_{\alpha}\)
& \(\mathbf{P17}\)\\
\midrule

\(\mathrm{TP}09_{\alpha,\beta}\)
&
\(\begin{cases}
\mathbf{P03}^{-1/(\beta-\alpha)}, & \alpha\neq\beta,\\
\mathcal Z, & \alpha=\beta.
\end{cases}\)\\
\midrule

\(\mathrm{TP}10_{\alpha}\)
& \(\mathcal Z\)\\
\midrule

\(\mathrm{TP}11_{\alpha,\beta,\gamma},\
  \mathrm{TP}14_{\alpha,\beta,\gamma}\)
&
Let \(\Delta=\gamma^{2}+4\beta\).\newline
\(\begin{cases}
\mathbf{P02},
  & \gamma=2\alpha,\ \beta=-\alpha^{2},\\
\mathbf{P08}^{(0,\,1/(2\alpha-\gamma))},
  & \beta=\alpha(\alpha-\gamma),\ \gamma\neq2\alpha,\\
\mathbf{P08}^{(1,\,2/(2\alpha-\gamma))},
  & \Delta=0,\ \gamma\neq2\alpha,\\
\mathbf{P08}^{(u,v)},
  & \beta\neq\alpha(\alpha-\gamma),\ \Delta\neq0,
\end{cases}\)\newline
where
\(\displaystyle
\begin{aligned}
u&=\frac{2\alpha-\gamma-\sqrt{\Delta}}
        {2\alpha-\gamma+\sqrt{\Delta}},
&
v&=\frac{2}{2\alpha-\gamma+\sqrt{\Delta}}.
\end{aligned}\)\\
\midrule

\(\mathrm{TP}12_{\alpha,\beta,\gamma}\)
&
{\scriptsize
\(\begin{cases}
\mathbf{P02},
  & \alpha=\beta=\gamma,\\
\mathbf{P08}^{(0,\,1/(\alpha-\gamma))},
  & \alpha=\beta,\ \alpha\neq\gamma,\\
\mathbf{P08}^{(0,\,1/(\alpha-\beta))},
  & \alpha=\gamma,\ \alpha\neq\beta,\\
\mathbf{P08}^{(1,\,1/(\alpha-\beta))},
  & \beta=\gamma,\ \alpha\neq\beta,\\
\mathbf{P08}^{((\alpha-\gamma)/(\alpha-\beta),\,1/(\alpha-\beta))},
  & (\alpha-\beta)(\alpha-\gamma)\neq0,\ \beta\neq\gamma.
\end{cases}\)}\\
\midrule

\(\mathrm{TP}13_{\alpha,\beta},\
  \mathrm{TP}15_{\alpha,\beta}\)
&
\(\begin{cases}
\mathbf{P07}^{1/(\alpha-\beta)}, & \alpha\neq\beta,\\
\mathcal Z, & \alpha=\beta.
\end{cases}\)\\
\midrule

\(\mathrm{TP}16_{\alpha,\beta}\text{--}
  \mathrm{TP}19_{\alpha,\beta}\)
&
\(\begin{cases}
\mathbf{P14}^{1/(\alpha-\beta)}, & \alpha\neq\beta,\\
\mathcal Z, & \alpha=\beta.
\end{cases}\)\\
\midrule

\(\mathrm{TP}20_{\alpha,\beta,\gamma}\)
&
\(\begin{cases}
\mathbf{P15}^{2/(\beta-\alpha)}, & \alpha\neq\beta,\\
\mathbf{P01}, & \alpha=\beta,\ \gamma\neq1,\\
\mathcal Z, & \alpha=\beta,\ \gamma=1.
\end{cases}\)\\
\midrule

\(\mathrm{TP}21_{\alpha,\beta}\)
&
\(\begin{cases}
\mathbf{P15}^{2/(\beta-\alpha)}, & \alpha\neq\beta,\\
\mathcal Z, & \alpha=\beta.
\end{cases}\)\\
\midrule

\(\mathrm{TP}22_{\alpha,\beta}^{\,(\beta\neq0)},\
  \mathrm{TP}24_{\alpha}\)
& \(\mathbf{P01}\)\\
\midrule

\(\mathrm{TP}23_{\alpha,\beta}\)
&
\(\begin{cases}
\mathbf{P15}^{2/(\alpha-\beta)}, & \alpha\neq\beta,\\
\mathcal Z, & \alpha=\beta.
\end{cases}\)\\
\midrule

\(\mathrm{TP}25_{\alpha,\beta,\gamma}\)
& \(\mathcal Z\)\\
\midrule

\(\mathrm{TP}26_{\alpha,\beta,\gamma},\
  \mathrm{TP}27_{\alpha,\beta,\gamma}\)
&
\(\begin{cases}
\mathbf{P12}^{-1/(\beta-\alpha)}, & \alpha\neq\beta,\\
\mathcal Z, & \alpha=\beta.
\end{cases}\)\\
\midrule

\(\mathrm{TP}28_{\alpha,\beta},\
  \mathrm{TP}29_{\alpha,\beta}\)
& \(\mathcal Z\)\\
\midrule

\(\mathrm{TP}30_{\alpha,\beta}\)
&
\(\begin{cases}
\mathbf{P13}, & \beta\neq1,\\
\mathcal Z, & \beta=1.
\end{cases}\)\\
\midrule

\(\mathrm{TP}31_{\alpha,\beta}\)
& \(\mathbf{P13}\)\\
\midrule

\(\mathrm{TP}32_{\alpha,\beta}\)
& \(\mathcal Z\)\\
\midrule

\(\mathrm{TP}33_{\alpha,\beta}\)
&
\(\begin{cases}
\mathbf{P12}^{0}, & \beta\neq1,\\
\mathcal Z, & \beta=1.
\end{cases}\)\\
\midrule

\(\mathrm{TP}34_{\alpha},\ \mathrm{TP}35_{\alpha}\)
& \(\mathbf{P12}^{0}\)\\
\midrule

\(\mathrm{TP}36_{\alpha}\text{--}\mathrm{TP}39_{\alpha}\)
& \(\mathcal Z\)\\
\midrule

\(\mathrm{TP}40_{\alpha,\beta}^{\,(\alpha\neq\beta)}\)
&
\(\begin{cases}
\mathbf{P10}^{(\alpha-1)/(\beta-1)}, & \beta\neq1,\\
\mathbf{P11}, & \beta=1.
\end{cases}\)\\
\midrule

\(\mathrm{TP}41_{\alpha}\)
&
\(\begin{cases}
\mathbf{P05}, & \alpha=1,\\
\mathbf{P10}^{1}, & \alpha\neq1.
\end{cases}\)\\
\midrule

\(\mathrm{TP}42_{\alpha}\)
&
\(\begin{cases}
\mathcal Z, & \alpha=1,\\
\mathbf{P06}, & \alpha\neq1.
\end{cases}\)\\
\midrule

\(\mathrm{TP}43_{\alpha,\beta}^{\,(\alpha\neq0)}\)
&
\(\begin{cases}
\mathbf{P06}, & \beta=0,\\
\mathbf{P11}, & \beta=1,\\
\mathbf{P10}^{-1/(\beta-1)}, & \beta\neq0,\ 1.
\end{cases}\)\\
\midrule

\(\mathrm{TP}44_{\alpha,\beta}^{\,(\beta\neq0)}\)
&
Let \(\Delta=\alpha^{2}+4\beta\).\newline
\(\begin{cases}
\mathbf{P03}^{0}, & \alpha=2,\ \beta=-1,\\
\mathbf{P09}^{0}, & \beta=1-\alpha,\ \alpha\neq2,\\
\mathbf{P09}^{1}, & \Delta=0,\ \alpha\neq2,\\
\mathbf{P09}^{u}, & \beta\neq1-\alpha,\ \Delta\neq0,
\end{cases}\)\newline
where
\(\displaystyle
u=\frac{\alpha-2-\sqrt{\Delta}}
        {\alpha-2+\sqrt{\Delta}}.\)\\
\midrule

\(\mathrm{TP}45_{\alpha,\beta}\)
&
\(\begin{cases}
\mathbf{P09}^{1}, & \beta=0,\\
\mathbf{P09}^{0}, & \beta=1,\\
\mathbf{P09}^{-1/(\beta-1)}, & \beta\neq0,\ 1.
\end{cases}\)\\
\midrule

\(\mathrm{TP}46_{\alpha}^{\,(\alpha\neq0)}\)
& \(\mathbf{P10}^{1-\alpha}\)\\
\midrule

\(\mathrm{TP}47_{\alpha}\)
&
\(\begin{cases}
\mathbf{P06}, & \alpha=0,\\
\mathbf{P10}^{1}, & \alpha\neq0.
\end{cases}\)\\
\midrule

\(\mathrm{TP}48_{\alpha}\)
& \(\mathbf{P09}^{0}\)\\
\midrule

\(\mathrm{TP}49_{\alpha}\)
&
\(\begin{cases}
\mathbf{P06}, & \alpha=1,\\
\mathbf{P11}, & \alpha=0,\\
\mathbf{P10}^{1/\alpha}, & \alpha\neq0,\ 1.
\end{cases}\)\\
\midrule

\(\mathrm{TP}50\)
& \(\mathbf{P10}^{1}\)\\
\midrule

\(\mathrm{TP}51_{\alpha}\)
& \(\mathbf{P10}^{0}\)\\
\midrule

\(\mathrm{TP}52_{\alpha}\)
&
Let
\(\displaystyle
\lambda_{1}=\frac{\alpha+\sqrt{\alpha^{2}+4}}{2}\)
and
\(\displaystyle
\lambda_{2}=\frac{\alpha-\sqrt{\alpha^{2}+4}}{2}\).\newline
\(\begin{cases}
\mathbf{P09}^{1}, & \alpha^{2}+4=0,\\
\mathbf{P09}^{\lambda_{2}/\lambda_{1}},
  & \alpha^{2}+4\neq0.
\end{cases}\)\\
\midrule

\(\mathrm{TP}53,\ \mathrm{TP}54\)
& \(\mathbf{P05}\)\\
\midrule

\(\mathrm{TP}55_{\alpha}\)
& \(\mathbf{P03}^{0}\)\\
\midrule

\(\mathrm{TP}56\)
& \(\mathbf{P11}\)\\
\midrule

\(\mathrm{TP}57,\ \mathrm{TP}58_{\alpha}\)
& \(\mathcal Z\)\\
\midrule

\(\mathrm{TP}59_{\alpha},\ \mathrm{TP}60_{\alpha}\)
& \(\mathbf{P15}^{0}\)\\
\midrule

\(\mathrm{TP}61_{\alpha}^{\,(\alpha\neq1,2)},\
  \mathrm{TP}63_{\alpha},\ \mathrm{TP}64\)
& \(\mathbf{P10}^{2}\)\\
\midrule

\(\mathrm{TP}62_{\alpha}\)
&
\(\begin{cases}
\mathbf{P05}, & \alpha\neq0,\\
\mathcal Z, & \alpha=0.
\end{cases}\)\\
\midrule

\(\mathrm{TP}65_{\alpha,\beta}\)
&
\(\begin{cases}
\mathbf{P05}, & \beta\neq0,\\
\mathcal Z, & \beta=0.
\end{cases}\)\\
\midrule

\(\mathrm{TP}66_{\alpha}\)
&
\(\begin{cases}
\mathbf{P11}, & \alpha\neq1,\\
\mathcal Z, & \alpha=1.
\end{cases}\)\\
\midrule

\(\mathrm{TP}67_{\alpha}\)
& \(\mathbf{P11}\)\\
\midrule

\(\mathrm{TP}68_{\alpha}\)
& \(\mathcal Z\)\\
\midrule

\(\mathrm{TP}69_{\alpha}\)
& \(\mathbf{P10}^{2}\)\\
\midrule

\(\mathrm{TP}70_{\alpha},\
  \mathrm{TP}71_{\alpha}^{\,(\alpha\neq-1)}\)
&
\(\begin{cases}
\mathbf{P05}, & \alpha\neq1,\\
\mathcal Z, & \alpha=1.
\end{cases}\)\\
\midrule

\(\mathrm{TP}72_{\alpha}\)
& \(\mathbf{P05}\)\\
\midrule

\(\mathrm{TP}73_{\alpha}^{\,(\alpha\neq0)}\)
&
\(\begin{cases}
\mathbf{P11}, & \alpha\neq1,\\
\mathcal Z, & \alpha=1.
\end{cases}\)\\
\midrule

\(\mathrm{TP}74_{\alpha}\)
& \(\mathbf{P11}\)\\
\end{longtable}

\endgroup

\begin{proof}
By Theorem~\ref{thm:commutator-construction}, every algebra in Theorem~\ref{thm:classification-np3} induces a transposed Poisson algebra through \eqref{eq:commutator}. Proposition~\ref{prop:isomorphism-preserved} shows that it is enough to consider the representatives in that theorem.

Let \(\{e_1,e_2,e_3\}\) and \(\{f_1,f_2,f_3\}\) be the source and target bases, respectively. We write \(\phi=(v_1,v_2,v_3)\) when \(\phi(e_i)=v_i\) for \(i=1,2,3\). By \cite{Zhang-Z}, an invertible linear map
\(\phi\colon S\to T\) is an isomorphism if and only if
\begin{equation}\label{eq:TP-isomorphism-test}
\begin{aligned}
\phi(x\cdot_S y)&=\phi(x)\cdot_T\phi(y),\\
\phi([x,y]_{\circ,S})&=[\phi(x),\phi(y)]_T
\end{aligned}
\end{equation}
for all basis elements \(x,y\) of \(S\). We compute the induced brackets and compare the resulting algebras with the normal forms in
\cite{Zhang-Z}, using the isomorphism-reduction and isomorphism-testing methods in
\cite{Zhang-Z}.

We first record the sign-sensitive parameters in the \(\mathbf{P15}\) cases. For \(\mathrm{TP}21_{\alpha,\beta}\), we have
\[
[e_1,e_3]_{\circ}=\frac{\beta-\alpha}{2}e_1,
\qquad
[e_2,e_3]_{\circ}=(\beta-\alpha)e_2.
\]
Put \(c=\beta-\alpha\). If \(c\neq0\), then
\[
\phi=\left(f_1-f_2,\ f_2,\ \frac c2f_3\right)
\]
gives
\[
\mathrm{TP}21_{\alpha,\beta}
\cong\mathbf{P15}^{2/c}.
\]
Thus, the corresponding parameter is \(2/(\beta-\alpha)\). If \(c=0\), the induced bracket is zero.

For \(\mathrm{TP}23_{\alpha,\beta}\), we have
\[
[e_1,e_3]_{\circ}=\frac{\alpha-\beta}{2}e_1,
\qquad
[e_2,e_3]_{\circ}=(\alpha-\beta)e_2.
\]
Put \(d=\alpha-\beta\). If \(d\neq0\), then
\[
\phi=\left(f_1-f_2,\ f_2,\ \frac d2f_3\right)
\]
gives
\[
\mathrm{TP}23_{\alpha,\beta}
\cong\mathbf{P15}^{2/d}.
\]
Thus, the corresponding parameter is \(2/(\alpha-\beta)\). If \(d=0\), the induced bracket is zero.

For \(u\neq0\), direct substitution into
\eqref{eq:TP-isomorphism-test} gives
\[
\begin{alignedat}{2}
\mathbf{P09}^{u}
&\cong\mathbf{P09}^{u^{-1}},
&\qquad
\phi&=(u^2f_1,u^2(u-1)f_1+u^3f_2,uf_3),\\[2pt]
\mathbf{P08}^{(u,v)}
&\cong\mathbf{P08}^{(u^{-1},v/u)},
&
\phi&=(f_1,(u-1)f_1+uf_2,uf_3).
\end{alignedat}
\]
Consequently, the two choices of each square root in the statement determine isomorphic classes.

Computing the induced bracket for each source family and solving \eqref{eq:TP-isomorphism-test} gives Table~\ref{tab:correspondence}. For reasons of space, and because many of the induced transposed Poisson algebras in the table are mutually isomorphic, we do not list a change of basis for every entry. Instead, we give representative explicit isomorphisms for the target classes and exceptional parameter values occurring in the table. Any parameter omitted from a source label is arbitrary.

The classes \(\mathbf{P01}\)--\(\mathbf{P05}\) are represented by
\[
\begin{alignedat}{2}
\mathrm{TP}24
&\cong\mathbf{P01},
&\qquad \phi&=(-f_1,f_3,f_2),\\[2pt]
\mathrm{TP}12_{0,0,0}
&\cong\mathbf{P02},
&\phi&=(-f_3,f_1,f_2),\\[2pt]
\mathrm{TP}09_{0,-1/a}
&\cong\mathbf{P03}^{a}\quad(a\neq0),
&\phi&=(f_2,f_1,af_3),\\[2pt]
\mathrm{TP}55
&\cong\mathbf{P03}^{0},
&\phi&=(f_2,f_1,-f_3),\\[2pt]
\mathrm{TP}05_{1/a,0}
&\cong\mathbf{P04}^{a}\quad(a\neq0),
&\phi&=(f_1,af_3,f_2),\\[2pt]
\mathrm{TP}54
&\cong\mathbf{P05},
&\phi&=(f_1,f_3,f_2).
\end{alignedat}
\]

The classes \(\mathbf{P06}\)--\(\mathbf{P09}\) are represented by
\[
\begin{alignedat}{2}
\mathrm{TP}42_0
&\cong\mathbf{P06},
&\qquad \phi&=(f_3,f_1,f_2),\\[2pt]
\mathrm{TP}13_{1/a,0}
&\cong\mathbf{P07}^{a}\quad(a\neq0),
&\phi&=(f_1,f_2,a^{-1}f_3),\\[2pt]
\mathrm{TP}11_{1/b,0,(1-a)/b}
&\cong\mathbf{P08}^{(a,b)}\quad(b\neq0),
&\phi&=(f_1,-b^{-1}f_2,b^{-1}f_3),\\[2pt]
\mathrm{TP}45_{0,1-1/a}
&\cong\mathbf{P09}^{a}\quad(a\neq0),
&\phi&=(a^{-1}f_3,a^{-2}f_1,-a^{-3}f_2),\\[2pt]
\mathrm{TP}48
&\cong\mathbf{P09}^{0},
&\phi&=(-f_3,f_1,f_1+f_2).
\end{alignedat}
\]
In the \(\mathbf{P08}^{(a,b)}\) row, \(a\) is arbitrary.

The classes \(\mathbf{P10}\)--\(\mathbf{P12}\) are represented by
\[
\begin{alignedat}{2}
\mathrm{TP}40_{1-a,0}
&\cong\mathbf{P10}^{a}\quad(a\neq1),
&\qquad \phi&=(f_3,f_2,(1-a)f_1+f_2),\\[2pt]
\mathrm{TP}41_0
&\cong\mathbf{P10}^{1},
&\phi&=(f_3,f_2,-f_1),\\[2pt]
\mathrm{TP}67
&\cong\mathbf{P11},
&\phi&=(f_3,f_1+f_2,f_2),\\[2pt]
\mathrm{TP}26_{0,-1/b,0}
&\cong\mathbf{P12}^{b}\quad(b\neq0),
&\phi&=(b^{-1}f_3,f_2,f_1+f_2),\\[2pt]
\mathrm{TP}35
&\cong\mathbf{P12}^{0},
&\phi&=(f_2,f_3,-f_1-f_2).
\end{alignedat}
\]

The classes \(\mathbf{P13}\)--\(\mathbf{P15}\) are represented by
\[
\begin{alignedat}{2}
\mathrm{TP}31
&\cong\mathbf{P13},
&\qquad \phi&=(-f_2,f_2+f_3,f_1+f_2),\\[2pt]
\mathrm{TP}17_{1/b,0}
&\cong\mathbf{P14}^{b}\quad(b\neq0),
&\phi&=(f_1+2f_2,f_1+f_2,b^{-1}f_3),\\[2pt]
\mathrm{TP}20_{0,2/b,1-1/b}
&\cong\mathbf{P15}^{b}\quad(b\neq0),
&\phi&=(f_1,f_2,b^{-1}f_3),\\[2pt]
\mathrm{TP}59
&\cong\mathbf{P15}^{0},
&\phi&=(f_1,f_2,f_3).
\end{alignedat}
\]

Finally, let \(i^2=-1\). The classes
\(\mathbf{P16}\)--\(\mathbf{P18}\) are represented by
\[
\begin{alignedat}{2}
\mathrm{TP}01_{i,0}
&\cong\mathbf{P16},
&\qquad \phi&=(-if_1+f_3,2f_2,if_1+f_3),\\[2pt]
\mathrm{TP}08_1
&\cong\mathbf{P17},
&\phi&=(f_3,f_1,f_2),\\[2pt]
\mathrm{TP}02_1
&\cong\mathbf{P18},
&\phi&=(f_3,f_2,f_1).
\end{alignedat}
\]

Under the stated parameter restrictions, every displayed triple is a basis of the target algebra, and direct substitution verifies \eqref{eq:TP-isomorphism-test}. For the remaining source families, direct computation of the induced brackets and the same isomorphism-testing equations give the correspondences recorded in Table~\ref{tab:correspondence}. The cases belonging to \(\mathcal Z\) follow directly from \eqref{eq:commutator}, since their induced brackets are zero.
\end{proof}

By Theorem~\ref{thm:tp-correspondence}, the only 3-dimensional transposed Poisson algebras that cannot be obtained from the classified Novikov--Poisson algebras with nonzero commutative associative product are
\[
\mathbf{P04}^{0},\qquad
\mathbf{P07}^{0},\qquad
\mathbf{P08}^{(\alpha,0)},\qquad
\mathbf{P19}.
\]
\begin{thm}\label{thm:three-dimensional-realization}
Every 3-dimensional complex transposed Poisson algebra except \(\mathbf{P19}\) admits a Novikov--Poisson realization. The class \(\mathbf{P19}\) has zero commutative product and underlying Lie algebra \(\mathfrak{sl}_2(\mathbb C)\).
\end{thm}

\begin{proof}
If the Lie bracket is zero, define \(x\circ y=0\) for all \(x,y\). We may therefore assume that the bracket is nonzero. By the preceding discussion, only
$
\mathbf{P04}^{0},\quad
\mathbf{P07}^{0},\quad
\mathbf{P08}^{(\alpha,0)}$ and $
\mathbf{P19}
$
remain to be considered.

For the first 3 classes, retain the zero commutative product and define, with all omitted products equal to zero,
\[
\begin{array}{c@{\qquad}l}
\mathbf{P04}^{0}
& e_1\circ e_2=e_3,\\[2pt]
\mathbf{P07}^{0}
& e_3\circ e_1=-e_1,\quad e_3\circ e_2=-e_2,\\[2pt]
\mathbf{P08}^{(\alpha,0)}
& e_3\circ e_1=-e_1-e_2,\quad
  e_3\circ e_2=-\alpha e_2.
\end{array}
\]
A direct verification shows that these products satisfy the Novikov identities. Since the commutative product is zero, the Novikov--Poisson compatibility identities hold automatically. Their commutator brackets are,
respectively,
\[
[e_1,e_2]=e_3,
\]
\[
[e_1,e_3]=e_1,\qquad [e_2,e_3]=e_2,
\]
and
\[
[e_1,e_3]=e_1+e_2,\qquad [e_2,e_3]=\alpha e_2.
\]
Thus they induce precisely \(\mathbf{P04}^{0}\), \(\mathbf{P07}^{0}\), and \(\mathbf{P08}^{(\alpha,0)}\).

If \(\mathbf{P19}\) admitted a Novikov--Poisson realization, then \(\mathfrak{sl}_2(\mathbb C)\) would be the commutator Lie algebra of a 3-dimensional Novikov algebra. However, the commutator Lie algebra of every finite-dimensional Novikov algebra over \(\mathbb C\) is solvable \cite{Burde}, whereas \(\mathfrak{sl}_2(\mathbb C)\) is not solvable. This contradiction completes the proof.
\end{proof}

\subsection{Sartayev's conjecture on the speciality of transposed Poisson algebras as GD algebras}
\label{sec:gd_conjecture}
This subsection recalls the notions of Gelfand--Dorfman algebras and special Gelfand--Dorfman algebras, states Sartayev's conjecture, and proves that it holds in dimension 3.

\begin{defn}\cite{Gelfand-Dorfman}
A \emph{Gelfand--Dorfman algebra}, abbreviated to \emph{GD algebra}, is a triple $(A,\circ,[\,,])$ such that $(A,\circ)$ is a Novikov algebra, $(A,[\,,])$ is a Lie algebra, and
\begin{equation}
\label{eq:GD_compatibility}
 [x,y\circ z]-[z,y\circ x]+[y,x]\circ z-[y,z]\circ x-y\circ[x,z]=0
\end{equation}
for all $x,y,z\in A$.
\end{defn}

\begin{defn}\cite{Kolesnikov-Sartayev-Orazgaliev}
A GD algebra \(A=(A,\circ,[\cdot,\cdot])\) is said to be special if it can be embedded into a GD algebra induced by a differential Poisson algebra. That is, there exists a differential Poisson algebra \((P,\cdot,\{\cdot,\cdot\},d)\) such that \(A\) is isomorphic to a GD subalgebra of \(P^{(d)}\), where \(P^{(d)}\) is the GD algebra on \(P\) defined by
$
        [u,v]=\{u,v\}$ and $
        u\circ v=u\cdot d(v)
      $ for all $ u,v\in P.
$
\end{defn}

If $(A,\cdot,[\,,])$ is a transposed Poisson algebra, then
\[
 x\circ y=x\cdot y
\]
is a Novikov product and $(A,\circ,[\,,])$ is a Gelfand--Dorfman algebra \cite{Kolesnikov}.  It is therefore natural to ask whether this Gelfand--Dorfman algebra is special.  Sartayev formulated the following conjecture \cite{Sartayev}.

\begin{conjecture}\label{conj:special-GD}
Every transposed Poisson algebra, regarded as a Gelfand--Dorfman algebra by putting $x\circ y=x\cdot y$, is special.
\end{conjecture}

The conjecture remains open in general.  We now prove that it holds in dimension 3.
\begin{thm}\label{thm:conjecture-proof}
Every complex transposed Poisson algebra of dimension 3 is special as a Gelfand--Dorfman algebra.
\end{thm}

\begin{proof}
Let $(A,\cdot,[\,,])$ be a 3-dimensional complex transposed Poisson algebra, regarded as a Gelfand--Dorfman algebra with $x\circ y=x\cdot y$.  If $A\not\cong\mathbf P19$, then Theorem~\ref{thm:three-dimensional-realization} shows that $A$ can be obtained from a Novikov--Poisson algebra.  Hence $A$ is special as a
Gelfand--Dorfman algebra.

It remains to consider $\mathbf P19$.  Its commutative product, and hence its Gelfand--Dorfman Novikov product, is zero, while its underlying Lie algebra is $\mathfrak{sl}_2(\mathbb C)$.  Every Lie algebra equipped with
the zero Novikov product is a special Gelfand--Dorfman algebra \cite{Kolesnikov-Sartayev}.  Therefore, $\mathbf P19$ is special, which completes the proof.
\end{proof}

\begin{rem}

Sartayev's conjecture holds in dimensions 2 and 3. We have also examined some 4-dimensional cases and found that all the transposed Poisson algebras considered are special as Gelfand--Dorfman algebras. We therefore believe that the conjecture holds in arbitrary dimension.

\end{rem}


\begin{thebibliography}{99}



\bibitem{AbdelwahabFernandezKaygorodov2025}
H.~Abdelwahab, A.~Fern\'andez Ouaridi and I.~Kaygorodov, Degenerations of Poisson-type algebras,
\emph{Rend. Circ. Mat. Palermo (2)} \textbf{74} (2025), no.~1, Paper No.~63, 43 pp.

\bibitem{AbdelwahabKaygorodovLubkovDerived2026}
H.~Abdelwahab, I.~Kaygorodov and R.~Lubkov, The algebraic and geometric classification of derived Jordan and
bicommutative algebras,
\emph{J. Pure Appl. Algebra} \textbf{230} (2026), no.~5, Paper No.~108252.

\bibitem{AbdelwahabKaygorodovLubkovRight2026}
H.~Abdelwahab, I.~Kaygorodov and R.~Lubkov, The algebraic and geometric classification of right alternative and semi-alternative algebras,
\emph{J. Algebra} \textbf{687} (2026), 792--824.

\bibitem{Abdurasulov-Adashev-Eshmeteva}
K.~Abdurasulov, J.~Adashev and S.~Eshmeteva, Transposed Poisson structures on solvable Lie algebras with filiform
nilradical, \emph{Commun. Math.} \textbf{32} (2024), no.~3, 441--483.

\bibitem{AlvarezKaygorodov2021}
M.~A.~Alvarez and I.~Kaygorodov, The algebraic and geometric classification of nilpotent weakly associative and symmetric Leibniz algebras, \emph{J. Algebra} \textbf{588} (2021), 278--314.

\bibitem{Bai}
C.~Bai, R.~Bai, L.~Guo and Y.~Wu, Transposed Poisson algebras, Novikov--Poisson algebras and 3-Lie algebras, \emph{J. Algebra} \textbf{632} (2023), 535--566.

\bibitem{Balinskii-Novikov}
A.~A.~Balinskii and S.~P.~Novikov, Poisson brackets of hydrodynamic type, Frobenius algebras and Lie algebras, \emph{Soviet Math. Dokl.} \textbf{32} (1985), no.~1, 228--231.

\bibitem{BaoHong}
N.~Bao and Y.~Hong, Algebraic constructions for Novikov--Poisson algebras, \emph{Mathematics} \textbf{10} (2022), no.~17, Paper No.~3158, 18 pp.

\bibitem{Burde}
D.~Burde and K.~Dekimpe, Novikov structures on solvable Lie algebras, \emph{J. Geom. Phys.} \textbf{56} (2006), no.~9, 1837--1855.




\bibitem{BeitesFernandezKaygorodov}
P.~Damas Beites, A.~Fern\'andez Ouaridi and I.~Kaygorodov,
The algebraic and geometric classification of transposed Poisson algebras, \emph{Rev. R. Acad. Cienc. Exactas F\'is. Nat. Ser. A Mat. RACSAM} \textbf{117} (2023), no.~2, Paper No.~55, 25 pp.

\bibitem{BeitesStructures}
P.~Damas Beites, B.~L.~M.~Ferreira and I.~Kaygorodov, Transposed Poisson structures, \emph{Results Math.} \textbf{79} (2024), no.~2, Paper No.~93, 29 pp.

\bibitem{Dorfman1993}
I.~Dorfman, \emph{Dirac Structures and Integrability of Nonlinear Evolution Equations}, Nonlinear Science: Theory and Applications, John Wiley \& Sons, Ltd., Chichester, 1993, xii+176 pp.

\bibitem{FernandezOuaridi2024}
A.~Fern\'andez Ouaridi, On the simple transposed Poisson algebras and Jordan superalgebras, \emph{J. Algebra} \textbf{641} (2024), 173--198.

\bibitem{FernandezOuaridi2026}
A.~Fern\'andez Ouaridi, On simple transposed Poisson algebras, arXiv:2604.26115, 19 pp.

\bibitem{Ferreira}
B.~L.~M.~Ferreira, I.~Kaygorodov and V.~Lopatkin, \(\frac{1}{2}\)-derivations of Lie algebras and transposed Poisson algebras,
\emph{Rev. R. Acad. Cienc. Exactas F\'is. Nat. Ser. A Mat. RACSAM} \textbf{115} (2021), no.~3, Paper No.~142, 19 pp.

\bibitem{Gelfand-Dorfman}
I.~M.~Gel'fand and I.~Ya.~Dorfman, Hamiltonian operators and algebraic structures related to them, \emph{Funct. Anal. Appl.} \textbf{13} (1979), no.~4, 248--262.

\bibitem{Gubarev-Sartayev}
V.~Gubarev and B.~K.~Sartayev, Free special Gelfand--Dorfman algebra, \emph{J. Algebra Appl.} \textbf{23} (2024), no.~14, Paper No.~2550005, 11 pp.

\bibitem{JinHong2026}
J.~Jin and Y.~Hong, Nilpotency and Frattini theory for transposed Poisson algebras, \emph{J. Algebra Appl.} (2026), Paper No.~2850014.

\bibitem{Kaygorodov-Block}
I.~Kaygorodov and M.~Khrypchenko, Transposed Poisson structures on Block Lie algebras and superalgebras,
\emph{Linear Algebra Appl.} \textbf{656} (2023), 167--197.

\bibitem{KaygorodovKhrypchenkoLopes2020}
I.~Kaygorodov, M.~Khrypchenko and S.~A.~Lopes, The algebraic and geometric classification of nilpotent anticommutative algebras, \emph{J. Pure Appl. Algebra} \textbf{224} (2020), no.~8,
Paper No.~106337, 32 pp.

\bibitem{KaygorodovKhrypchenkoPopov2021}
I.~Kaygorodov, M.~Khrypchenko and Y.~Popov, The algebraic and geometric classification of nilpotent terminal algebras, \emph{J. Pure Appl. Algebra} \textbf{225} (2021), no.~6, Paper No.~106625, 41 pp.

\bibitem{Kaygorodov-Lopatkin-Zhang}
I.~Kaygorodov, V.~Lopatkin and Z.~Zhang, Transposed Poisson structures on Galilean and solvable Lie algebras, \emph{J. Geom. Phys.} \textbf{187} (2023),
Paper No.~104781, 13 pp.

\bibitem{Kolesnikov}
P.~S.~Kolesnikov and A.~A.~Nesterenko, Conformal envelopes of Novikov--Poisson algebras, \emph{Sib. Math. J.} \textbf{64} (2023), no.~3, 598--610.



\bibitem{Kolesnikov-Sartayev}
P.~S.~Kolesnikov and B.~K.~Sartayev, On the special identities of Gelfand--Dorfman algebras, \emph{Exp. Math.} \textbf{33} (2024), no.~1, 165--174.

\bibitem{Kolesnikov-Sartayev-Orazgaliev}
P.~S.~Kolesnikov, B.~Sartayev and A.~Orazgaliev, Gelfand--Dorfman algebras, derived identities, and the Manin product of operads, \emph{J. Algebra} \textbf{539} (2019), 260--284.

\bibitem{LaraiedhSilvestrov2021}
I.~Laraiedh and S.~Silvestrov, Transposed Hom--Poisson and Hom--pre-Lie Poisson algebras, \emph{Hacet. J. Math. Stat.} \textbf{55} (2026), no.~2, 473--501.

\bibitem{MaLi2023}
T.~Ma and B.~Li, Transposed BiHom--Poisson algebras, \emph{Comm. Algebra} \textbf{51} (2023), no.~2, 528--551.

\bibitem{Rakhimov}
I.~S.~Rakhimov, I.~M.~Rikhsiboev and W.~Basri, Complete lists of low dimensional complex associative algebras,
arXiv:0910.0932, 8 pp.

\bibitem{Sartayev}
B.~K.~Sartayev, Some generalizations of the variety of transposed Poisson algebras, \emph{Commun. Math.} \textbf{32} (2024), no.~2, 55--62.

\bibitem{SunZhuLi}
Q.~Sun, Q.~Zhu and Z.~Li, Bialgebras and related algebras for Novikov--Poisson algebras, \emph{Comm. Algebra} \textbf{53} (2025), no.~5, 2049--2065.

\bibitem{Xu1996}
X.~Xu, On simple Novikov algebras and their irreducible modules, \emph{J. Algebra} \textbf{185} (1996), no.~3, 905--934.

\bibitem{Xu1997}
X.~Xu, Novikov--Poisson algebras, \emph{J. Algebra} \textbf{190} (1997), no.~2, 253--279.

\bibitem{Xu2000}
X.~Xu, Quadratic conformal superalgebras, \emph{J. Algebra} \textbf{231} (2000), no.~1, 1--38.

\bibitem{Xu2002}
X.~Xu, Construction of Gel'fand--Dorfman bialgebras from classical \(R\)-matrices, arXiv:math/0208177, 10 pp.

\bibitem{Yuan}
L.~Yuan and Q.~Hua, \(\frac{1}{2}\)-(bi)derivations and transposed Poisson algebra structures on Lie algebras,
\emph{Linear Multilinear Algebra} \textbf{70} (2022), no.~22, 7672--7701.

\bibitem{ZakharovSimple2023}
A.~S.~Zakharov, Simple Novikov--Poisson algebras, \emph{Sib. Electron. Math. Rep.} \textbf{20} (2023), no.~2, 1396--1404.

\bibitem{Zhang-Z}
Z.~Zhang, Z.~Hao, Y.~Sun and L.~Chen, The classification of 3-dimensional transposed Poisson algebras,
Preprint (2026), ResearchGate: 404124642, 19 pp.

\bibitem{Zhang4D}
Z.~Zhang, Z.~Hao, Y.~Sun and L.~Chen, The classification of 4-dimensional transposed Poisson algebras,
Preprint (2026), ResearchGate: 400517733, 21 pp.

\bibitem{ZhangGD2026}
Z.~Zhang, Z.~Hao, Y.~Sun and L.~Chen, Gelfand--Dorfman algebras: nilpotency, solvability, construction and
classification, arXiv:2607.09672, 35 pp.

\bibitem{GDCohomology}
Z.~Zhang, Z.~Hao, Y.~Sun and L.~Chen, Cohomology, deformations and inducibility problems for Gelfand--Dorfman algebras, Preprint (2026), ResearchGate: 413601978, 22pp.


\bibitem{ZhaoBaiMeng}
Y.~Zhao, C.~Bai and D.~Meng, Some results on Novikov--Poisson algebras, \emph{Int. J. Theor. Phys.} \textbf{43} (2004), no.~2, 519--528.

\end{thebibliography}
\end{document}